\documentclass[a4paper, 12pt]{amsart}
\usepackage{amssymb, hyperref, tikz}

\newcommand\field[1]{\mathbb{#1}}

\newcommand\FF{\field{F}}
\newcommand\NN{\field{N}}

\newcommand\TT{\field{T}}
\newcommand\ZZ{\field{Z}}
\newcommand\Bb{B}

\newcommand\Kk{\mathcal K}

\newcommand\Oo{\mathcal O}
\newcommand\Tt{\mathcal T}
\newcommand\Xx{\mathcal X}
\newcommand\Yy{\mathcal Y}

\renewcommand\ker{\operatorname{ker}}

\newcommand\Obj{\operatorname{Obj}}
\newcommand\Ext{\operatorname{Ext}}

\newcommand\id{\operatorname{id}}
\newcommand\cod{\operatorname{cod}}
\newcommand\dom{\operatorname{dom}}

\newcommand\lsp{\operatorname{span}}
\newcommand\clsp{\operatorname{\overline{span\!}\,\,}}
\newcommand\MCE{\operatorname{MCE}}

\newcommand\Hom{\operatorname{Hom}}

\newcommand\range{\operatorname{range}}

\newcommand\bal[1]{{#1} *_{d,s} {#1}}

\newcommand\Sfilters{T}
\newcommand\Smin{S}
\newcommand{\Suniv}{s}
\newcommand{\tT}{\widetilde{T}}
\newcommand{\BalAlg}[1]{\mathcal{B}_{#1}}
\newcommand{\Bmin}[1]{\mathcal{B}_{#1}^{\min{}}}
\newcommand\Csmin[1]{C^*_{\min{}}(#1)}

\theoremstyle{plain}
\newtheorem{theorem}{Theorem}[section]
\newtheorem*{theorem*}{Theorem}
\newtheorem*{prop*}{Proposition}
\newtheorem{cor}[theorem]{Corollary}
\newtheorem{lemma}[theorem]{Lemma}
\newtheorem{prop}[theorem]{Proposition}

\theoremstyle{remark}
\newtheorem{rmk}[theorem]{Remark}

\newtheorem{example}[theorem]{Example}

\theoremstyle{definition}
\newtheorem{dfn}[theorem]{Definition}

\newtheorem{notation}[theorem]{Notation}
\newtheorem*{notation*}{Notation}

\numberwithin{equation}{section}

\title
[$P$-graph algebras]
{Co-universal $C^*$-algebras associated to  generalised graphs}
\author{Nathan Brownlowe}
\address{N. Brownlowe\\ School of Mathematics and Applied Statistics\\ Austin Keane Building (15)\\ University of Wollongong\\ NSW 2522\\ AUSTRALIA}
\email{nathanb@uow.edu.au}
\author{Aidan Sims}
\address{A. Sims\\ School of Mathematics and Applied Statistics\\ Austin Keane Building (15)\\ University of Wollongong\\ NSW 2522\\ AUSTRALIA}
\email{asims@uow.edu.au}
\author{Sean T. Vittadello}
\email{sean.vittadello@gmail.com}
\date{\today}
\thanks{This research was supported by the Australian Research Council.}

\begin{document}

\begin{abstract}
We introduce $P$-graphs, which are generalisations of directed
graphs in which paths have a degree in a semigroup $P$ rather
than a length in $\NN$. We focus on semigroups $P$ arising as
part of a quasi-lattice ordered group $(G,P)$ in the sense of
Nica, and on $P$-graphs which are finitely aligned in the sense
of Raeburn and Sims. We show that each finitely aligned
$P$-graph admits a $C^*$-algebra $\Csmin{\Lambda}$ which is
co-universal for partial-isometric representations of $\Lambda$
which admit a coaction of $G$ compatible with the $P$-valued
length function. We also characterise when a homomorphism
induced by the co-universal property is injective. Our results
combined with those of Spielberg show that every Kirchberg
algebra is Morita equivalent $\Csmin{\Lambda}$ for some $(\NN^2
\ast \NN)$-graph $\Lambda$.
\end{abstract}

\maketitle

\section{Introduction}

The Cuntz-Krieger algebras $\Oo_A$ introduced in \cite{CK1980}
provide an extensive array of purely infinite simple
$C^*$-algebras. The study of these algebras has led in
particular to the celebrated Kirchberg-Phillips classification
theorem which says, roughly, that every purely infinite simple
$C^*$-algebra (these are now called Kirchberg algebras) is
determined up to isomorphism by its $K$-theory
\cite{Phillips2000}, and that for every pair of abelian groups
$G,H$, there exists a purely infinite simple $C^*$-algebra $A$
with $K_*(A) = (G,H)$.

However, not every purely infinite simple $C^*$-algebra is a
Cuntz-Krieger algebra: the results of \cite{CK1980a} imply that
the $K$-groups of a Cuntz-Krieger algebra are finitely
generated and have equal rank, and that the $K_1$-group is free
abelian. Graph $C^*$-algebras \cite{KPR1998, KPRR1997} and
their higher-rank analogues \cite{KP2000} were developed in
part to seek Cuntz-Krieger-like models for the remaining purely
infinite simple $C^*$-algebras.

This program has met with mixed success. On the one hand, graph
algebras themselves do not suffice to describe all Kirchberg
algebras: the results of \cite{RS2004} and
\cite{Szymanski2002a} imply that a purely infinite simple
$C^*$-algebra can be realised up to Morita equivalence as a
graph algebra if and only if its $K_1$-group is free abelian.
And the question of whether every purely infinite simple
$C^*$-algebra can be realised as a $k$-graph $C^*$-algebra
remains open. On the other hand, since higher-rank graph $C^*$-algebras include, in particular,
all finite tensor products of graph $C^*$-algebras
\cite[Corollary~3.5(iv)]{KP2000}, for every pair of abelian
groups $G,H$, there exist $2$-graphs $\Lambda_G$ and
$\Lambda_H$ such that each of $C^*(\Lambda_G)$ and
$C^*(\Lambda_H)$ is simple and purely infinite, and
$K_*(\Lambda_G) = (G,\{0\})$ while $K_*(\Lambda_H) = (\{0\},
H)$. In \cite{Spielberg2007, Spielberg2007b} Spielberg
developed a construction which incorporates $\Lambda_G$ and
$\Lambda_H$ in a kind of hybrid graph $\Lambda$ in such a way
that the $C^*$-algebra associated to $\Lambda$ is itself simple
and purely infinite and has $K$-theory $(G,H)$. So every purely
infinite simple $C^*$-algebra can be realised up to stable
isomorphism as the $C^*$-algebra of one of Spielberg's hybrid
graphs, and so can, in a sense, be built from $k$-graph algebras.

A particularly powerful source of intuition when dealing with
graph $C^*$-algebras and $k$-graph $C^*$-algebras is that each
$k$-graph $C^*$-algebra can be realised up to Morita
equivalence as a crossed product of an AF algebra by an action
of $\ZZ^k$ \cite{KP2003}. It is therefore natural to seek an
analogous description of Spielberg's models. While Spielberg's
construction does not lend itself immediately to such a
description, the discussion of \cite[Examples~1.5]{FS2002}
suggests that one may be able to think of Spielberg's hybrid
graphs as generalised $k$-graphs in which the degree functor
from $\Lambda$ to $\NN^k$ has been replaced by a functor taking
values in the free product $\NN^2
* \NN$. The results of \cite{pp_CLSV2009} then suggest that the purely
infinite simple $C^*$-algebra associated to a hybrid graph can
be regarded as a crossed product of an AF core by $\NN^2 * \NN$.

In this paper, we introduce the notion of a $P$-graph
(Definition~\ref{dfn:P-graph}) for a quasi-lattice ordered
group $(G,P)$ in the sense of Nica, and associate to each
$P$-graph $\Lambda$ a $C^*$-algebra $\Csmin{\Lambda}$. We show
that Spielberg's hybrid graphs can be regarded as $(\NN^2 *
\NN)$-graphs, and that the associated $(\NN^2 * \NN)$-graph
$C^*$-algebra as constructed in this paper coincides with the
purely infinite simple $C^*$-algebra associated to the hybrid
graph by Spielberg (Theorem~\ref{isom of algebras}). In
particular, the class of $P$-graph algebras contains, up to
Morita equivalence, every Kirchberg algebra.

Our approach to the construction of the $P$-graph algebra
associated to a $P$-graph $\Lambda$ does not follow the
traditional lines used for graphs and $k$-graphs in the
literature (see, for example, \cite{BHRS2002, KP2000,
KPRR1997}). Instead we proceed using the notion of a
co-universal $C^*$-algebra. This approach was inspired by
Katsura's description of the $C^*$-algebras he associates to
Hilbert bimodules \cite[Proposition~7.14]{Katsura2007}, and was
applied in \cite{pp_CLSV2009} to product systems. Our main
result, Theorem~\ref{thm:couniversal alg of Lambda}, says that
every finitely aligned $P$-graph $\Lambda$ admits a
$C^*$-algebra which is co-universal for representations of
$\Lambda$ which are nonzero on generators and carry a natural
coaction of $G$.

Co-universal properties have been explored previously as a
means of specifying $C^*$-algebras associated to directed
graphs \cite{pp_SW2009}. However, this approach is relatively
new, and one of our motivations for tackling $P$-graphs in this
way is to develop techniques for establishing the existence of
a co-universal algebra for a given system of generators and
relations. In particular, we address in Examples
\ref{eg:SY}~and~\ref{eg:CLSV} the problems arising in previous
approaches to co-universal algebras detailed in
\cite[Example~3.9]{pp_CLSV2009} and
\cite[Example~3.16]{pp_SY2007}. Our other motivation for using
co-universal properties is that we deal here with groups which
need not be amenable. Since unitary representations of the
groups themselves are, in some instances, examples of our
construction, one cannot expect to obtain a $C^*$-algebra which
satisfies a version of the gauge-invariant uniqueness theorem
as a universal $C^*$-algebra (see
\cite[Remark~5.4]{pp_CLSV2009}).


The notion of a representation of a $P$-graph $\Lambda$ and the
associated universal $C^*$-algebra $\Tt C^*(\Lambda)$ were
introduced in \cite{RS2005}. The algebra $\Tt C^*(\Lambda)$ is
generated by partial isometries $\{\Suniv_\mu : \mu \in
\Lambda\}$ and is spanned by the elements of the form
$\Suniv_\mu \Suniv_\nu^*$ such that $s(\mu) = s(\nu)$. The
fixed-point algebra for the canonical coaction of $G$ on $\Tt
C^*(\Lambda)$ is the subalgebra spanned by the elements
$\Suniv_\mu \Suniv_\nu^*$ such that $d(\mu) = d(\nu)$, where $d
: \Lambda \to P$ denotes the generalised length function.
Analysing this fixed-point subalgebra of $\Tt C^*(\Lambda)$ is the
traditional first step in establishing a uniqueness
theorem for $\Tt C^*(\Lambda)$ itself.

Our innovation in this paper is to begin by developing an
analysis of the universal $C^*$-algebra $\BalAlg{\Lambda}$
generated by partial isometries $\{\omega_{\mu,\nu} : d(\mu) =
d(\nu), s(\mu) = s(\nu)\}$ satisfying the same relations as the
$\Suniv_\mu \Suniv^*_\nu$. In particular, we characterise in
Theorem~\ref{thm:ideals in core} the ideals of
$\BalAlg{\Lambda}$ which contain none of the
$\omega_{\mu,\nu}$. We then use this analysis to construct a
$C^*$-algebra $\Bmin{\Lambda}$ which is co-universal for
representations of $\BalAlg{\Lambda}$ by nonzero partial
isometries (Theorem~\ref{thm:Bmin couniversal}). We prove that
$\omega_{\mu,\nu} \mapsto \Suniv_\mu \Suniv^*_\nu$ determines
an isomorphism of $\BalAlg{\Lambda}$ with the fixed-point
algebra in $\Tt C^*(\Lambda)$. We are then able to use the
categorical approach to coactions studied in \cite{EKQR2006} to
construct a $C^*$-algebra $\Csmin{\Lambda}$ which is generated
by a representation of $\Lambda$ by nonzero partial isometries $\{\Smin_\lambda : \lambda \in
\Lambda\}$, and carries a normal coaction of $G$ whose
fixed-point algebra coincides with $\Bmin{\Lambda}$. We present
a bootstrapping argument employing the canonical conditional
expectations associated to coactions and the universal property
of $\Tt C^*(\Lambda)$ to deduce from the co-universal property of $\Bmin{\Lambda}$ that $\Csmin{\Lambda}$ is
co-universal for representations of $\Lambda$ by nonzero
partial isometries which preserve the canonical coaction of
$G$. A key tool in our analysis of $\Bmin{\Lambda}$ is Exel's
use of filters and ultrafilters as a tool for studying
representations of inverse semigroups. Example~\ref{eg:SY} and
Remark~\ref{rmk:SY} highlight the advantage of this approach.

\section{Preliminaries}\label{sec:prelims}

Following Nica \cite{Nica1992}, we say that $(G,P)$ is a
\emph{quasi-lattice ordered group} if $G$ is a discrete group,
$P$ is a subsemigroup of $G$ such that $P \cap P^{-1} = \{ e
\}$, and, under the partial order $p \le q \Leftrightarrow
p^{-1} q \in P$ on $G$, every pair of elements $p, q \in G$
with a common upper bound in $P$ has a least common upper bound
$p \vee q$ in $P$. We write $p \vee q = \infty$ to indicate
that $p,q \in G$ have no common upper bound in $P$, and we
write $p \vee q < \infty$ otherwise.

\begin{dfn}\label{dfn:P-graph}
Let $(G,P)$ be a quasi-lattice ordered group. A
\emph{$P$-graph} $(\Lambda, d)$ consists of a countable
category $\Lambda = ( \Obj (\Lambda), \Hom (\Lambda), \cod, \dom)$
together with a functor $d \colon \Lambda \to P$, called the
\emph{degree map}, which satisfies the \emph{factorisation
property}: for every $\lambda \in \Lambda$ and $p, q \in P$
with $d(\lambda) = pq$ there exist unique elements $\mu, \nu
\in \Lambda$ such that $\lambda = \mu \nu$, $d(\mu) = p$, and
$d(\nu) = q$.
\end{dfn}

\begin{notation}
Let $(G,P)$ be a quasi-lattice ordered group, and let $\Lambda$
be a $P$-graph. For $p \in P$ we define
\[
\Lambda^{p} := \{\, \lambda \in \Lambda \colon d(\lambda) = p \,\}.
\]
The factorisation property implies that $\Lambda^0 = \{\id_o :
o \in \Obj (\Lambda)\}$. We define surjections $r,s : \Lambda
\to \Lambda^0$ by $r(\lambda) := \id_{\cod(\lambda)}$ and
$s(\lambda) := \id_{\dom(\lambda)}$, and we regard $\Lambda^0$
as the vertex set of $\Lambda$.

For $E \subset \Lambda$ and $\lambda \in \Lambda$ we define
\[
\lambda E := \{\, \lambda \mu \colon \mbox{$\mu \in E$ and $r(\mu)=s(\lambda)$} \,\}
\]
and
\[
E \lambda := \{\, \mu \lambda \colon \mbox{$\mu \in E$ and $s(\mu)=r(\lambda)$} \,\}.
\]
Hence, for $E \subset \Lambda$ and $v \in \Lambda^{0}$,
\[
vE = \{\, \mu \in E \colon r(\mu)=v \,\}
\]
and
\[
Ev := \{\, \mu \in E \colon s(\mu)=v \,\}.
\]
\end{notation}

We write $\Lambda *_s \Lambda$ for the set $\{(\mu,\nu) \in
\Lambda \times \Lambda : s(\mu) = s(\nu)\}$, and write
$\bal{\Lambda}$ for the set $\{(\mu,\nu) \in \Lambda *_s
\Lambda : d(\mu) = d(\nu)\}$ consisting of pairs which are
balanced with respect to the degree functor. More generally,
for any pair $U,V$ of subsets of $\Lambda$, we will write $U
*_s V$ for $(U \times V) \cap (\Lambda *_s \Lambda)$, and $U
*_{d,s} V$ for $(U \times V) \cap (\bal{\Lambda})$.

\medskip

\begin{dfn}
Let $(G,P)$ be a quasi-lattice ordered group and let $\Lambda$
be a $P$-graph. For $\mu, \nu \in \Lambda$ we say that $\lambda
\in \Lambda$ is a \emph{minimal common extension} of $\mu$ and
$\nu$ if $d(\mu)\vee d(\nu)<\infty$, $d(\lambda) = d(\mu) \vee d(\nu)$ and there exist
$\alpha \in \Lambda^{d(\mu)^{-1} (d(\mu) \vee d(\nu))}$ and
$\beta \in \Lambda^{d(\nu)^{-1} (d(\mu) \vee d(\nu))}$ such
that $\lambda = \mu \alpha = \nu \beta$. We write
$\MCE(\mu,\nu)$ for the set of minimal common extensions of
$\mu$ and $\nu$. We say that $\Lambda$ is \emph{finitely
aligned} if $\MCE(\mu,\nu)$ is finite (possibly empty) for all
$\mu, \nu \in \Lambda$. Given $v \in \Lambda^{0}$ we say that $E \subset
v \Lambda$ is \emph{exhaustive} if for every
$\mu \in v \Lambda$ there exists $\lambda \in E$ such that
$\MCE(\mu,\lambda) \ne \emptyset$.
\end{dfn}

Note that in particular, if $d(\mu) \vee d(\nu) = \infty$, then
$\MCE(\mu,\nu) = \emptyset$.

\begin{notation}\label{ntn:compacts}
We make frequent use of the abstract $C^*$-algebras generated
by matrix units indexed by countable sets. Fix a countable set
$X$. By \cite[Corollary~A.9 and Remark~A.10]{Raeburn2005},
there is a unique (up to canonical isomorphism) $C^*$-algebra
$\Kk_X$ generated by elements $\{\Theta_{x,y} : x,y \in X\}$
satisfying
\begin{equation}\label{eq:MU relations}
\Theta_{x,y}^* = \Theta_{y,x} \qquad\text{ and }\qquad
\Theta_{x,y}\Theta_{w,z} = \begin{cases} \Theta_{x,z} &\text{ if $y = w$} \\ 0 &\text{ otherwise.}\end{cases}
\end{equation}
We call a family satisfying \eqref{eq:MU relations} a
\emph{family of matrix units over} $X$. In particular, given
two such families $\{\alpha_{x,y} : x,y \in X\}$ and
$\{\beta_{x,y} : x,y \in X\}$, there is a unique isomorphism
$C^*(\{\alpha_{x,y} : x,y \in X\}) \to C^*(\{\beta_{x,y} : x,y
\in X\})$ which carries each $\alpha_{x,y}$ to $\beta_{x,y}$.
Since the set $\{\Theta_{x,y}:x,y\in X\}$ is closed under
adjoints and multiplication, $\Kk_X=\clsp\{\Theta_{x,y}:x,y\in
X\}$.

For a finite subset $F$ of $X$, write $P_F \in \Kk_X$ for the
projection $P_F := \sum_{x \in F} \Theta_{x,x}$. An $\varepsilon/3$ argument shows that the net
$\{P_F : F \subset X \text{ is finite}\}$ is an approximate
identity for $\Kk_X$.
\end{notation}

\begin{notation*}
In this paper, given a finitely aligned $P$-graph $\Lambda$, we
deal both with representations of $\bal{\Lambda}$, and also
with representations of $\Lambda$ itself. In addition, in each
case there are two distinguished representations
--- the universal representation and the co-universal
representation --- which we frequently wish to talk about.

Our convention will be that Greek letters are used to denote
representations of $\bal{\Lambda}$, and Roman letters are used
to denote representations of $\Lambda$; and the universal and
co-universal representations will be denoted by the same letter
in lower case and upper case respectively.
\end{notation*}

\section{Filters and ultrafilters in $P$-graphs}\label{sec:filters}

In the theories of graph $C^*$-algebras and of $k$-graph
$C^*$-algebras, spaces of infinite paths --- or an appropriate
analogue --- are often used to construct a representation by
nonzero partial isometries. Precisely what should constitute an
infinite path in a $P$-graph is not immediately clear; in fact,
the question is already complicated enough for $k$-graphs. In
this section we show how the roles played by paths and infinite
paths in the representation theory of $k$-graphs can be played
by filters and ultrafilters in the setting of $P$-graphs. We
show how initial segments can be appended to or removed from
filters and ultrafilters, and use this construction to
associate to each $P$-graph a specific family of partial
isometries on Hilbert space which will later be a key
ingredient in our construction of the co-universal algebra of
the $P$-graph. We took the idea of using ultrafilters to obtain
a minimal representation from Exel who introduced it in the
context of partial-isometric representations of inverse
semigroups \cite{Exel2009}.

Let $(G,P)$ be a quasi-lattice ordered group, and let $\Lambda$
be a finitely aligned $P$-graph. We define a relation $\preceq$
on $\Lambda$ by $\lambda \preceq \mu$ if and only if $\mu =
\lambda\mu'$ for some $\mu' \in \Lambda$.

\begin{dfn}\label{dfn:filters}
A \emph{filter} of $\Lambda$ is a nonempty subset $U$ of
$\Lambda$ such that
\begin{itemize}
\item[(F1)] if $\mu \in U$ and $\lambda \preceq \mu$ then
    $\lambda \in U$, and
\item[(F2)] if $\mu,\nu \in U$, then there exists $\lambda
    \in U$ such that $\mu,\nu \preceq \lambda$.
\end{itemize}
\end{dfn}
Fix a filter $U$ of $\Lambda$. The factorisation property and~(F2) imply that if $\mu,\nu \in U$, then there is a
unique element $\lambda$ of $\MCE(\mu,\nu)$ such that $\lambda
\in U$. This combined with~(F1) and that $U$ is nonempty implies that there is a unique $v \in \Lambda^0$ such
that $v \in U$, and then we have $r(\lambda) = v$ for all
$\lambda \in U$. We write $r(U) = v$.

We write $\widehat{\Lambda}$ for the collection of all filters
of $\Lambda$, and we regard $\widehat{\Lambda}$ as a partially ordered set under inclusion.
An \emph{ultrafilter} of $\Lambda$ is a filter $U \in
\widehat{\Lambda}$ which is maximal; that is, $U$ is not
properly contained in any other filter $V$ of $\Lambda$. We
write $\widehat{\Lambda}_\infty$ for the collection of all
ultrafilters of $\Lambda$.

\begin{lemma}\label{lem:plenty of ultrafilters}
Let $(G,P)$ be a quasi-lattice ordered group, and let $\Lambda$
be a finitely aligned $P$-graph. For each $\lambda \in \Lambda$
there exists an ultrafilter $U$ of $\Lambda$ such that $\lambda
\in U$.
\end{lemma}
\begin{proof}
We aim to apply Zorn's lemma. Let $\Xx_\lambda$ denote the
collection of all filters $U$ of $\Lambda$ such that $\lambda
\in U$. Observe that $\Xx_\lambda$ is nonempty because $\{\mu
\in \Lambda : \lambda \in \mu\Lambda\}$ is a filter of
$\Lambda$ which contains $\lambda$.

Fix a totally ordered subset $\Yy$ of $\Xx_\lambda$. We claim
that $\bigcup \Yy$ is an upper bound for $\Yy$ in
$\Xx_\lambda$. To see this, it suffices to show that $\bigcup
\Yy$ is a filter of $\Lambda$; that is, we must verify
(F1)~and~(F2). For~(F1), suppose that $\mu \in \bigcup \Yy$ and
$\lambda \preceq \mu$. By definition of $\bigcup \Yy$, we have
$\mu \in V$ for some filter $V \in \Yy$; and then since $V$ is
a filter we have $\lambda \in V \subset \bigcup \Yy$ also.
For~(F2), suppose that $\mu,\nu \in \bigcup \Yy$. Then there
exist $V, W \in \Yy$ such that $\mu \in V$ and $\nu \in W$.
Since $\Yy$ is totally ordered, we may suppose without loss of
generality that $V \subset W$. So $\mu,\nu \in W$, and since
$W$ is a filter, it follows that there exists $\lambda \in W
\subset \bigcup \Yy$ such that $\mu,\nu \preceq W$. Hence
$\bigcup \Yy$ is an upper bound for $\Yy$ as claimed.

Zorn's Lemma now implies that $\Xx_\lambda$
has a maximal element $U$. We have $\lambda \in U$ by
definition of $\Xx_\lambda$. To see that $U$ is an ultrafilter,
observe that if $V$ is a filter with $U \subset V$, then
$\lambda \in U \subset V$ forces $V \in \Xx_\lambda$, and since
$U$ is maximal in $\Xx_\lambda$, it follows that $U = V$.
\end{proof}

\begin{lemma}\label{lem:ultrafilters-FEsets}
Let $(G,P)$ be a quasi-lattice ordered group, and let $\Lambda$
be a finitely aligned $P$-graph. Let $\mu \in \Lambda$, and let
$E$ be a finite exhaustive subset of $s(\mu)\Lambda$. Let $U$
be an ultrafilter of $\Lambda$ such that $\mu \in U$. Then
there exists $\alpha \in E$ such that $\mu\alpha \in U$.
\end{lemma}
\begin{proof}
We first claim that there exists $\alpha \in E$ such that
$\MCE(\mu\alpha,\lambda) \not= \emptyset$ for all $\lambda \in
U$. To see this, suppose for contradiction that for each
$\alpha\in E$ there exists $\lambda_\alpha \in U$ such that
$\MCE(\mu\alpha,\lambda_\alpha) = \emptyset$. Since $U$ is a
filter, there exists $\lambda \in U$ such that $\mu \preceq
\lambda$ and $\lambda_\alpha \preceq \lambda$ for all $\alpha
\in E$. Fix $\alpha \in E$. Since
$\MCE(\mu\alpha,\lambda_\alpha) = \emptyset$ and $\lambda \in
\lambda_\alpha \Lambda$, we have $\MCE(\mu\alpha,\lambda) =
\emptyset$. Since $\mu \preceq \lambda$ we may factorise
$\lambda = \mu\lambda'$ and then since
\[
\mu\MCE(\alpha,\lambda') = \MCE(\mu\alpha,\mu\lambda') = \MCE(\mu\alpha,\lambda) = \emptyset,
\]
we have $\MCE(\alpha,\lambda') = \emptyset$. Since $\alpha \in
E$ was arbitrary, this contradicts that $E$ is exhaustive.

Fix $\alpha \in E$ such that $\MCE(\mu\alpha,\lambda) \not=
\emptyset$ for all $\lambda \in U$. We will show that
$\mu\alpha \in U$. Since $\Lambda$, and hence $U$, is countable
there is a cofinal sequence $(\lambda_n)^\infty_{n=1}$ in $U$
with $\lambda_1 = \mu$ and $\lambda_n \preceq \lambda_{n+1}$
for all $n$. We have $\MCE(\mu\alpha, \lambda_n) \not=
\emptyset$ for all $n$.

We claim that there exists a sequence $(\xi_n)^\infty_{n=1}$ in
$\Lambda$ such that for each $n$,
\begin{enumerate}
\item\label{it:in MCE} $\xi_n \in \MCE(\mu\alpha,
    \lambda_n)$,
\item\label{it:inceasing} if $n \ge 2$, then $\xi_{n-1}
    \preceq \xi_{n}$, and
\item\label{it:all the way up} $\MCE(\xi_n, \lambda_m)
    \not= \emptyset$ for all $m$.
\end{enumerate}
We prove the claim by induction on $n$. When $n = 1$, the path
$\xi_1 := \mu\alpha$ satisfies (\ref{it:in
MCE})~and~(\ref{it:all the way up}) by choice of $\alpha$,
and~(\ref{it:inceasing}) is trivial.

Now suppose that there are paths $\xi_1, \dots \xi_k$
satisfying (\ref{it:in MCE})--(\ref{it:all the way up}). For
each $m > k $, the set $\MCE(\xi_k, \lambda_m)$ is nonempty by
the inductive hypothesis, so we may fix $\eta_m \in \MCE(\xi_k,
\lambda_m)$. Fix $m > k$. Since $\lambda_{k+1} \preceq
\lambda_m$, there is a unique $\xi \in \MCE(\xi_k,
\lambda_{k+1})$ such that $\xi \preceq \eta_m$. Since $\Lambda$
is finitely aligned, $\MCE(\xi_k, \lambda_{k+1})$ is finite, so
there exists $\xi_{k+1} \in \MCE(\xi_k, \lambda_{k+1})$ such
that $\xi_{k+1} \preceq \eta_m$ for infinitely many $m > k$. We
claim that this $\xi_{k+1}$ satisfies (\ref{it:in
MCE})--(\ref{it:all the way up}). It is straightforward to see
that $\xi_{k+1}$ satisfies~(\ref{it:in MCE}) using that $\xi_k$
satisfies~(\ref{it:in MCE}), and that $\lambda_k \preceq
\lambda_{k+1}$. It satisfies~(\ref{it:inceasing}) by
definition. To see that it satisfies~(\ref{it:all the way up}),
fix $m \in \NN$. By choice of $\xi_{k+1}$ there exists $m' \ge
m$ such that $\xi_{k+1} \preceq \eta_{m'} \in \MCE(\mu\alpha,
\lambda_{m'})$ and then that $\lambda_m \preceq \lambda_{m'}$
forces $\MCE(\xi_{k+1}, \lambda_m) \not= \emptyset$ also. This
completes the proof of the claim.

Now let $V := \bigcup_{n \in \NN} \{\zeta \in \Lambda : \zeta \preceq \xi_n\}$. That the $\xi_n$ are increasing with respect to
$\preceq$ implies that $V$ is a filter.
Since the $\lambda_n$ are cofinal in $U$ and since each
$\lambda_n \preceq \xi_n$, we have $U \subset V$. Since $U$ is
an ultrafilter, it follows that $V = U$, and since $\mu\alpha
\preceq \xi_n\in V$ for all $n$, it follows that $\mu\alpha \in U$ as
claimed.
\end{proof}

Fix $\lambda \in \Lambda$. For $U \in \widehat{\Lambda}$ with
$r(U) = s(\lambda)$, we define
\[
\lambda \cdot U
    := \bigcup_{\mu\in U}\{\alpha \in \Lambda :\alpha\preceq\lambda\mu\}.
\]
For $V \in \widehat{\Lambda}$ such that $\lambda \in V$, we
define
\[
\lambda^*\cdot V := \{\mu \in \Lambda : \lambda\mu \in V\}.
\]

\begin{lemma}\label{action}
Let $(G,P)$ be a quasi-lattice ordered group, and let $\Lambda$
be a finitely aligned $P$-graph. Fix $\lambda \in \Lambda$, and
$U,V \in \widehat{\Lambda}$ with $r(U) = s(\lambda)$ and
$\lambda \in V$. Then
\begin{enumerate}
\item\label{it:action-preserves spectrum} $\lambda\cdot U$
    and $\lambda^*\cdot V$ belong to $\widehat{\Lambda}$;
\item\label{it:action-mutually inverse}
    $\lambda^*\cdot(\lambda\cdot U) = U$ and
    $\lambda\cdot(\lambda^*\cdot V)=V$; and
\item\label{it:action-preserves boundary} $U \in
    \widehat{\Lambda}_\infty \implies \lambda\cdot U \in
    \widehat{\Lambda}_\infty$, and $V \in
    \widehat{\Lambda}_\infty \implies \lambda^* \cdot V \in
    \widehat{\Lambda}_\infty$.
\end{enumerate}
\end{lemma}
\begin{proof}
(\ref{it:action-preserves spectrum})~Since $\lambda \in \lambda
\cdot U$ and $s(\lambda) \in \lambda^* \cdot V$, both $\lambda
\cdot U$ and $\lambda^* \cdot V$ are nonempty.

It is routine to use the factorisation property to check that
both $\lambda \cdot U$ and $\lambda^* \cdot V$ satisfy~(F1).
Suppose $\mu,\nu \in \lambda \cdot U$. Then there exist
$\alpha,\beta \in U$ such that $\mu\preceq \lambda\alpha$ and
$\nu \preceq \lambda\beta$. Since $U \in \widehat{\Lambda}$,
there exists $\eta \in U$ such that $\alpha,\beta \preceq
\eta$, and the factorisation property then forces $\mu,\nu
\preceq \lambda\eta \in \lambda \cdot U$. So $\lambda\cdot U$
satisfies~(F2). Now suppose that $\mu,\nu \in \lambda^* \cdot
V$. Then $\lambda\mu, \lambda\nu \in V$. Since $V \in
\widehat{\Lambda}$, there exists $\eta \in V$ such that
$\lambda\mu,\lambda\nu \preceq \eta$; it then follows from the
factorisation property that $\eta = \lambda\eta'$ for some
$\eta'$ with $\mu,\nu \preceq \eta'$, and we have $\eta' \in
\lambda^*\cdot V$ by definition. So $\lambda^*\cdot V$
satisfies (F2). This completes the proof of (1).

(\ref{it:action-mutually inverse})~We have
\[
\mu \in \lambda^* \cdot (\lambda \cdot U)
    \iff \lambda\mu \in \lambda \cdot U
    \iff \mu \in U,
\]
so $\lambda^*\cdot(\lambda\cdot U)=U$. To see that $\lambda\cdot(\lambda^*\cdot V)=V$, we first calculate
\begin{align*}
\mu \in \lambda \cdot (\lambda^* \cdot V)
    &\iff \mu \preceq \lambda\alpha\text{ for some $\alpha \in \lambda^*\cdot V$}\\
    &\iff \mu \preceq \lambda\alpha\text{ for some $\alpha$ with $\lambda\alpha \in V$}.
\end{align*}
Since $V$ satisfies~(F2) and $\lambda \in V$, every $\nu \in V$
satisfies $\nu \preceq \lambda\alpha$ for some $\lambda\alpha \in
V$. That is $\nu \in V$ if and only if $\nu \preceq \lambda\alpha$
for some $\lambda\alpha \in V$, and it follows that $\mu \in
\lambda \cdot (\lambda^* \cdot V)$ if and only if $\mu \in V$
as required.

(\ref{it:action-preserves boundary})~Suppose that $U \in
\widehat{\Lambda}_\infty$, and suppose that $U' \in
\widehat{\Lambda}$ with $\lambda\cdot U \subset U'$. We must
show that $U' = \lambda\cdot U$. We have $\lambda \in \lambda
\cdot U \subset U'$, so $\lambda^* \cdot U'$ makes sense. We
then have
\[
\lambda^* \cdot U' \supset \lambda^*\cdot(\lambda\cdot U) = U
\]
by part~(\ref{it:action-mutually inverse}). Since $U$ is an
ultrafilter, it follows that $\lambda^* \cdot U' = U$ and then
another application of~(\ref{it:action-mutually inverse}) gives
\[
U' = \lambda \cdot (\lambda^* \cdot U') = \lambda \cdot U.
\]
A similar argument shows that $\lambda^*\cdot V$ is an
ultrafilter.
\end{proof}

\begin{dfn}\label{dfn:Sfilters}
Let $(G,P)$ be a quasi-lattice ordered group, and let $\Lambda$
be a finitely aligned $P$-graph. Define $\Sfilters : \Lambda
\to \Bb(\ell^2(\widehat{\Lambda}))$ by $\Sfilters_\lambda e_U
:= \delta_{s(\lambda), r(U)} e_{\lambda \cdot U}$.
\end{dfn}

Routine calculations using the inner-product on
$\ell^2(\widehat{\Lambda})$ (see for
example~\cite[Proposition~2.12]{RSY2004}) show that the
$\Sfilters_\lambda$ are partial isometries with adjoints characterised by
\begin{equation}\label{eq:Smin adjoints}
\Sfilters^*_\lambda e_U = \begin{cases}
    e_{\lambda^* \cdot U} &\text{ if $\lambda \in U$} \\
    0 &\text{ otherwise}.
\end{cases}
\end{equation}
In particular, for each $\lambda \in \Lambda$, the operator
$\Sfilters_\lambda \Sfilters^*_\lambda$ is the orthogonal
projection onto the subspace $\clsp\{e_U : U \in
\widehat{\Lambda}, \lambda \in U\} \subset
\ell^2(\widehat{\Lambda})$.

\section{The balanced algebras of a $P$-graph}\label{sec:C0(G0)}

In this section we introduce and analyse what we call the
\emph{balanced algebras} of a $P$-graph. We associate to each
finitely aligned $P$-graph two balanced algebras --- a
universal balanced algebra, and a quotient thereof, which we
call the co-universal balanced algebra. For a $k$-graph, the
universal balanced algebra would correspond to the fixed-point
algebra for the gauge-action on the Toeplitz algebra of the
$k$-graph, and the co-universal balanced algebra to the
fixed-point algebra for the gauge action on the Cuntz-Krieger
algebra of the $k$-graph.

We show in the next section that the universal
balanced algebra of a $P$-graph is isomorphic to the
fixed-point algebra for the canonical coaction of $G$ on the
Toeplitz algebra of the $P$-graph. We then use this and a
bootstrapping argument to construct the co-universal algebra of
the $P$-graph. It turns out that the analysis
of \cite{pp_CLSV2009} is greatly simplified by first
demonstrating that the fixed-point subalgebra of the Toeplitz
algebra has a universal property and admits a co-universal quotient in its own right.

\subsection{The universal balanced algebra}\label{subsec:univ
bal}

In this subsection we define the universal balanced algebra of
a $P$-graph $\Lambda$ and characterise the representations of this balanced algebra.

\begin{dfn}
Let $(G,P)$ be a quasi-lattice ordered group, and let $\Lambda$
be a finitely aligned $P$-graph. A \emph{representation} of
$\bal{\Lambda}$ in a $C^*$-algebra $B$ is a map $\tau :
\bal{\Lambda} \to B$, $(\mu,\nu) \mapsto \tau_{\mu,\nu}$ such
that for all $(\mu,\nu), (\xi,\eta) \in \bal{\Lambda}$,
\begin{enumerate}
\item[(B1)] $\tau_{\mu,\nu}^* = \tau_{\nu,\mu}$, and
\item[(B2)] $\tau_{\mu,\nu}\tau_{\xi,\eta} =
    \sum_{\nu\alpha = \xi\beta \in \MCE(\nu,\xi)}
    \tau_{\mu\alpha,\eta\beta}$.
\end{enumerate}
We denote by $C^*(\tau)$ the $C^*$-subalgebra of $B$
generated by the $\tau_{\mu,\nu}$.
\end{dfn}

Observe that $\MCE(\mu,\mu) = \{\mu\}$ for all $\mu$. Hence
(B1)~and~(B2) imply that $\tau_{\mu,\mu} = \tau_{\mu,\mu}^2 =
\tau_{\mu,\mu}^*$, so each $\tau_{\mu,\mu}$ is a projection.
Moreover, $\tau_{\mu,\nu}^*\tau_{\mu,\nu} = \tau_{\nu,\nu}$ for
all $(\mu,\nu) \in \bal{\Lambda}$. So the range of a
representation of $\bal{\Lambda}$ consists of partial
isometries. Finally, since $\MCE(\mu,\nu) = \MCE(\nu,\mu)$ for
all $\mu$ and $\nu$, condition~(B2) implies that the
projections $\{\tau_{\mu,\mu} : \mu \in \Lambda\}$ pairwise
commute.

\begin{lemma}\label{lem:filter repn}
Let $(G,P)$ be a quasi-lattice ordered group, and let $\Lambda$
be a finitely aligned $P$-graph. Let $\Sfilters : \Lambda \to
\Bb(\ell^2(\widehat{\Lambda}))$ be as in
Definition~\textup{\ref{dfn:Sfilters}}. Then the map $(\mu,\nu) \mapsto
\Sfilters_\mu \Sfilters^*_\nu$ is a representation of
$\bal{\Lambda}$.
\end{lemma}
\begin{proof}
Condition~(B1) is obvious. For~(B2), fix $(\mu,\nu), (\xi,\eta)
\in \bal{\Lambda}$. Then for $U \in \widehat{\Lambda}$, we have
\[
\Sfilters_\nu \Sfilters^*_\nu \Sfilters_\xi \Sfilters^*_\xi e_U
    =\begin{cases}
        e_U &\text{ if $\nu,\xi \in U$}\\
        0 &\text{ otherwise.}
    \end{cases}
\]
As discussed after Definition~\ref{dfn:filters}, $\MCE(\nu,\xi)
\cap U$ has at most one element, so
\[
\sum_{\nu\alpha = \xi\beta \in \MCE(\nu,\xi)} T_{\nu\alpha} T_{\nu\alpha}^* e_U
    =\begin{cases}
        e_U &\text{ if $\MCE(\nu,\xi) \cap U \not= \emptyset$}\\
        0 &\text{ otherwise.}
    \end{cases}
\]
Conditions (F1)~and~(F2) imply that $\MCE(\nu,\xi) \cap U \not=
\emptyset$ if and only if $\nu,\xi \in U$. So $\Sfilters_\nu
\Sfilters^*_\nu \Sfilters_\xi \Sfilters^*_\xi = \sum_{\lambda
\in \MCE(\nu,\xi)} \Sfilters_\lambda \Sfilters^*_\lambda$.

Since the $\Sfilters_\lambda$ are all partial isometries, it
follows that
\[
\Sfilters_\mu \Sfilters^*_\nu \Sfilters_\xi \Sfilters^*_\eta
    = \Sfilters_\mu \Sfilters^*_\nu \Sfilters_\nu \Sfilters^*_\nu \Sfilters_\xi \Sfilters^*_\xi \Sfilters_\xi \Sfilters^*_\eta
    = \sum_{\nu\alpha=\xi\beta \in \MCE(\nu,\xi)} \Sfilters_\mu \Sfilters^*_\nu \Sfilters_{\nu\alpha} \Sfilters^*_{\xi\beta} \Sfilters_\xi \Sfilters^*_\eta
\]
So it is enough to fix $U \in \widehat{\Lambda}$ and show that
for $\nu\alpha = \xi\beta \in \MCE(\nu,\xi)$, we have
$\Sfilters_\mu \Sfilters^*_\nu \Sfilters_{\nu\alpha}e_U =
\Sfilters_{\mu\alpha} e_U$ (it will follow from symmetry that
$\Sfilters^*_{\xi\beta}\Sfilters_\xi \Sfilters^*_\eta e_U =
\Sfilters^*_{\eta\beta} e_U$).  If $r(U) \not = s(\alpha)$ then
both sides are equal to zero, so suppose that $r(U) =
s(\alpha)$. Then
\begin{align*}
\mu \cdot (\nu^* \cdot (\nu\alpha \cdot U))
    &= \{\kappa \in \Lambda : \kappa \preceq \mu\zeta\text{ for some } \zeta \in \nu^* \cdot (\nu\alpha \cdot U)\} \\
    &= \{\kappa \in \Lambda : \kappa \preceq \mu\zeta\text{ for some $\zeta$ such that }\nu\zeta \in \nu\alpha \cdot U\} \\
    &= \{\kappa \in \Lambda : \kappa \preceq \mu\zeta\text{ for some } \zeta \in \alpha \cdot U\} \\
    &= \{\kappa \in \Lambda : \kappa \preceq \mu\alpha\zeta'\text{ for some } \zeta' \in U\} \\
    &= \mu\alpha \cdot U.
\end{align*}
Hence
\[
\Sfilters_\mu \Sfilters^*_\nu \Sfilters_{\nu\alpha}e_U
    = e_{\mu \cdot (\nu^* \cdot (\nu\alpha \cdot U))}
    = e_{\mu\alpha \cdot U}
    = \Sfilters_{\mu\alpha}e_U.
\]

So $(\mu,\nu) \mapsto \Sfilters_\mu \Sfilters^*_\nu$ satisfies
(B1)~and~(B2) as required.
\end{proof}

\begin{prop}\label{prp:Bb generators nonzero}
Let $(G,P)$ be a quasi-lattice ordered group, and let $\Lambda$
be a finitely aligned $P$-graph. There exists a $C^*$-algebra
$\BalAlg{\Lambda}$ generated by a representation $\omega$ of
$\bal{\Lambda}$ which is universal in the following sense: for
every representation $\tau$ of $\bal{\Lambda}$, there is a
$C^*$-homomorphism $\rho_\tau : \BalAlg{\Lambda} \to C^*(\tau)$
satisfying $\rho_\tau \circ \omega = \tau$. Moreover, the
partial isometries $\omega_{\mu,\nu}$ are all nonzero, and for
each $\mu \in \Lambda$ and each finite exhaustive set $E
\subset s(\mu)\Lambda \setminus \{s(\mu)\}$,
\[\textstyle
\prod_{\alpha\in E}(\omega_{\mu,\mu} - \omega_{\mu\alpha,\mu\alpha}) \not= 0.
\]
\end{prop}
\begin{proof}
An argument along the lines of \cite[pages 12~and~13]{Raeburn2005} shows
that there is a $C^*$-algebra $\BalAlg{\Lambda}$ generated by a representation $\omega$ of $\bal{\Lambda}$ which is universal for representations of
$\bal{\Lambda}$.

To see that each $\omega_{\mu,\nu}$ is nonzero, observe that for
each $\lambda \in \Lambda$, there exists $U \in
\widehat{\Lambda}$ such that $\lambda \in U$, and hence the
partial isometries $\Sfilters_\lambda$ of
Definition~\ref{dfn:Sfilters} are all nonzero. It follows that
\begin{equation}\label{eq:balalg gens nonzero}
\Sfilters_\mu \Sfilters^*_\nu \not= 0 \qquad\text{ for all $(\mu,\nu) \in \bal{\Lambda}$.}
\end{equation}
By Lemma~\ref{lem:filter repn} and the universal property of $\BalAlg{\Lambda}$, there is a homomorphism which takes each
$\omega_{\mu,\nu}$ to $\Sfilters_\mu \Sfilters^*_\nu$, and it
follows that the $\omega_{\mu,\nu}$ are nonzero as well.

Fix $\mu \in \Lambda$ and a finite exhaustive set $E \subset
s(\mu)\Lambda \setminus \{s(\mu)\}$. Let $U_\mu := \{\lambda
\in \Lambda : \mu \in \lambda\Lambda\}$. Then $U_\mu \in
\widehat{\Lambda}$ and we have $\mu \in U_\mu$, but $\mu\alpha
\not\in U_\mu$ for all $\alpha \in E$. So $\Sfilters_\mu
\Sfilters^*_\mu e_{U_\mu} = e_{U_\mu}$, but
$\Sfilters_{\mu\alpha} \Sfilters^*_{\mu\alpha} e_{U_\mu} = 0$
for all $\alpha \in E$. Thus
\begin{equation}\label{eq:balalg gaps nonzero}
\prod_{\alpha\in E}(\Sfilters_\mu \Sfilters^*_\mu - \Sfilters_{\mu\alpha} \Sfilters^*_{\mu\alpha})e_{U_\mu} = e_{U_\mu} \not= 0.
\end{equation}
Lemma~\ref{lem:filter repn} and the first statement of this
proposition imply that there is a homomorphism taking
$\prod_{\alpha\in E}(\omega_{\mu,\mu}
-\omega_{\mu\alpha,\mu\alpha})$  to $\prod_{\alpha\in
E}(\Sfilters_\mu \Sfilters^*_\mu - \Sfilters_{\mu\alpha}
\Sfilters^*_{\mu\alpha})$, so $\prod_{\alpha\in
E}(\omega_{\mu,\mu} -\omega_{\mu\alpha,\mu\alpha})$ is nonzero
also.
\end{proof}

\begin{theorem}\label{thm:Bb uniqueness}
Let $(G,P)$ be a quasi-lattice ordered group, let $\Lambda$ be
a finitely aligned $P$-graph, and let $\tau$ be a
representation of $\bal{\Lambda}$. Then the homomorphism
$\rho_\tau : \BalAlg{\Lambda} \to C^*(\tau)$ induced by the
universal property of $\BalAlg{\Lambda}$ is injective if and
only if
\begin{enumerate}
\item\label{it:t lambda nonzero} $\tau_{\mu,\mu}\not=0$ for
    all $\mu\in\Lambda$, and
\item\label{it:gaps nonzero} $\prod_{\alpha\in
    E}(\tau_{\mu,\mu} - \tau_{\mu\alpha,\mu\alpha}) \not=
    0$ for each $\mu\in\Lambda$ and each finite exhaustive
    $E\subset s(\mu)\Lambda \setminus\{s(\mu)\}$.
\end{enumerate}
Moreover, $\BalAlg{\Lambda}$ is an AF algebra.
\end{theorem}

To prove the theorem, we first analyse the structure of
$\BalAlg{\Lambda}$. We require the notion of a $\vee$-closed
subset of $P$.

A subset $F$ of $P$ is called $\vee$-closed if, whenever $p,q \in F$ satisfy $p \vee q < \infty$,
we have $p \vee q \in F$. Since the $\vee$ operation is
both commutative and associative, given $G \subset P$ the
formula
\[
\vee G := p_1 \vee (p_2 \vee (p_3 \vee... \vee p_{|G|})) \in P \cup \{\infty\}
\]
is well-defined. So if $F \subset P$ is finite, then
\[
\overline{F} := \{\vee G : G \subset F, \vee G \not= \infty\}
\]
is a finite $\vee$-closed subset of $P$, which contains $F$
since $\vee\{p\} = \{p\}$ for all $p \in P$. It follows that
the collection of finite $\vee$-closed subsets of $P$ is
directed under $\subseteq$ and the union of all finite
$\vee$-closed subsets of $P$ is $P$ itself. A minimal element
of a finite $\vee$-closed subset $F$ of $P$ is an element $p
\in F$ such that $q \in F$ implies $q \not\le p$.

\begin{dfn}
For each $\vee$-closed subset $F$ of $P$, we define
\[
B_F := \Big\{\sum_{p\in F}a_p:a_p\in\clsp\{\omega_{\mu,\nu}:(\mu,\nu)\in\Lambda^p *_s \Lambda^p\}\text{ for each $p\in F$}\Big\},
\]
and for each $p \in P$, we write $B_p$ for $B_{\{p\}}$. So
$B_F$ consists of finite linear combinations of elements of the
$B_p$ where $p$ ranges over $F$.
\end{dfn}

For the following lemma, recall from Section~\ref{sec:prelims}
our notation for the abstract algebra $\Kk_X$ generated by
matrix units indexed by a countable set $X$.

\begin{lemma}\label{lem:Bb structure}
Let $(G,P)$ be a quasi-lattice ordered group, and let $\Lambda$
be a finitely aligned $P$-graph. Then
\begin{enumerate}
\item\label{it:B_p} for each $p \in P$ there is an
    isomorphism $B_p \cong \bigoplus_{v \in \Lambda^0}
    \Kk_{\Lambda^p v}$ satisfying $\omega_{\mu,\nu} \mapsto
    \Theta_{\mu,\nu}$;
\item\label{it:B_F closed} for each finite $\vee$-closed
    subset $F$ of $P$, the set $B_F$ is an AF
    $C^*$-subalgebra of $\BalAlg{\Lambda}$; and
\item\label{it:Bb directlim} $\BalAlg{\Lambda} = \varinjlim
    B_F$ where the collection of finite $\vee$-closed
    subsets of $P$ is directed by inclusion.
\end{enumerate}
\end{lemma}
\begin{proof}
For~(\ref{it:B_p}), one uses (B1)~and~(B2) to see that
$\{\omega_{\mu,\nu} : (\mu,\nu) \in \Lambda^p v \times
\Lambda^p v\}$ is a family of matrix units for each $v \in
\Lambda^0$. The $\omega_{\mu,\nu}$ are all nonzero by
Proposition~\ref{prp:Bb generators nonzero}. The uniqueness of
$\Kk_{\Lambda^p v}$ implies that $B_p(v) :=
\clsp\{\omega_{\mu,\nu} : (\mu,\nu) \in \Lambda^p v \times
\Lambda^p v\}$ is isomorphic to $\Kk_{\Lambda^p v}$ via
$\omega_{\mu,\nu} \mapsto \Theta_{\mu,\nu}$. Moreover,~(B2)
implies that if $\mu,\nu \in \Lambda^pv$ and $\xi,\eta \in
\Lambda^p w$ for distinct $v,w$, then $\omega_{\mu,\nu}
\omega_{\xi,\eta} = 0$. Hence $B_p = \bigoplus_{v \in
\Lambda^0} B_p(v) \cong \bigoplus_{v \in \Lambda^0}
\Kk_{\Lambda^p v}$.

For~(\ref{it:B_F closed}), we proceed by induction on $|F|$ as
in \cite[Lemma~3.6]{pp_CLSV2009}. When $|F| = 1$, statement (2) follows
from~(\ref{it:B_p}). Now suppose that $B_F$ is an AF
$C^*$-subalgebra of $\BalAlg{\Lambda}$ whenever $|F| \le k$,
and fix a $\vee$-closed subset $F$ of $P$ with $|F| = k+1$. Fix
a minimal element $m \in F$. Then $G := F \setminus \{m\}$ is
also $\vee$-closed. By the inductive hypothesis, $B_G$ is an AF
$C^*$-algebra. Moreover one can check on spanning elements
using~(B2) that $B_m B_G, B_G B_m \subset B_G$. Hence
\cite[Corollary~1.8.4]{Dixmier1977} implies that $B_F$ is a
$C^*$-algebra. To see that $B_F$ is AF, observe that $B_G$ is
an ideal of $B_F$ with quotient $B_F/B_G \cong B_{m}/(B_{m}
\cap B_G)$. Both $B_G$ and $B_m$ are AF by the inductive
hypothesis, and since quotients of AF algebras are AF, it
follows that $B_F$ is an extension of an AF algebra by an AF
algebra, and hence itself AF (see, for example,
\cite[Theorem~III.6.3]{Davidson1996}).

For~(\ref{it:Bb directlim}), observe that if $G \subset F$
are both $\vee$-closed, then $B_G \subset B_F$, and that
$\bigcup_F B_F$ contains all the generators of
$\BalAlg{\Lambda}$.
\end{proof}

We now establish two technical results which we shall use to
prove Theorem~\ref{thm:Bb uniqueness}.

\begin{lemma}\label{Lem: matrix units}
Let $(G,P)$ be a quasi-lattice ordered group, and let $\Lambda$
be a finitely aligned $P$-graph. Let $\tau$ be representation
of $\bal{\Lambda}$. Fix $v \in \Lambda^0$ and a finite subset
$H \subset v\Lambda\setminus\{v\}$. For $\mu,\nu \in \Lambda^p
v$, let
\[
\theta_{\mu,\nu} := \tau_{\mu,\nu}\prod_{\lambda \in H} (\tau_{\nu,\nu} - \tau_{\nu\lambda,\nu\lambda}).
\]
Then for $\mu,\nu, \rho,\sigma \in \Lambda^p v$,
\[
\theta^*_{\mu,\nu} = \theta_{\nu,\mu}\quad\text{ and }\quad\theta_{\mu,\nu}\theta_{\rho,\sigma} = \delta_{\nu,\rho} \theta_{\mu,\sigma}.
\]
In particular, if $\tau$ satisfies conditions \textup{(\ref{it:t lambda
nonzero})}~and~\textup{(\ref{it:gaps nonzero})} of Theorem~\textup{\ref{thm:Bb
uniqueness}}, then $\{\theta_{\mu,\nu} : \mu,\nu \in \Lambda^p
v\}$ is a family of nonzero matrix units.
\end{lemma}
\begin{proof}
For each $\mu,\nu\in\Lambda^pv$ we have
$(\tau_{\nu,\nu}-\tau_{\nu\lambda,\nu\lambda})\tau_{\nu,\mu}=\tau_{\nu,\mu}(\tau_{\mu,\mu}-\tau_{\mu\lambda,\mu\lambda})$,
and hence
\begin{equation}\label{eq: pull tau through}
\Big(\prod_{\lambda\in H}(\tau_{\nu,\nu}-\tau_{\nu\lambda,\nu\lambda})\Big)\tau_{\nu,\mu}
    = \tau_{\nu,\mu}\prod_{\lambda\in H}(\tau_{\mu,\mu}-\tau_{\mu\lambda,\mu\lambda}).
\end{equation}
It now follows from (B1)~and~(\ref{eq: pull tau through}) that
$\theta_{\mu,\nu}^*=\theta_{\nu,\mu}$. It follows from
(B2)~and~(\ref{eq: pull tau through}) that for
$\rho,\sigma\in\Lambda^pv$ we have
\[
\theta_{\mu,\nu}\theta_{\rho,\sigma}=
\begin{cases}
\big(\prod_{\lambda\in H}(\tau_{\mu,\mu}-\tau_{\mu\lambda,\mu\lambda})\big)
    \tau_{\mu,\sigma}\big(\prod_{\eta\in H}(\tau_{\sigma,\sigma}-\tau_{\sigma\eta,\sigma\eta})\big) & \text{if $\nu=\rho$,}\\
0 & \text{otherwise}.
\end{cases}
\]
Thus (\ref{eq: pull tau through}) for
$\mu,\sigma\in\Lambda^p v$ and that $\prod_{\eta\in
H}(\tau_{\sigma,\sigma}-\tau_{\sigma\eta,\sigma\eta})$ is a
projection imply that $\theta_{\mu,\nu}\theta_{\rho,\sigma} = \delta_{\nu,\rho} \theta_{\mu,\sigma}$. Hence $\{\theta_{\mu,\nu}:\mu,\nu\in\Lambda^pv\}$ is a family of matrix units.

Now suppose that $\tau$ satisfies conditions (\ref{it:t lambda
nonzero})~and~(\ref{it:gaps nonzero}) of Theorem~\ref{thm:Bb
uniqueness}. It suffices to show $\theta_{\mu,\mu}\not=0$ for
$\mu\in\Lambda^pv$. If $H$ is exhaustive, then in particular
$H\not=\emptyset$, and $\theta_{\mu,\mu}=\prod_{\lambda\in
H}(\tau_{\mu,\mu}-\tau_{\mu\lambda,\mu\lambda})\not= 0$ by
assumption. If $H$ is not exhaustive, then there exists
$\eta\in v\Lambda$ with $\MCE(\lambda,\eta)=\emptyset$ for all
$\lambda\in H$. It follows that
$\tau_{\mu\eta,\mu\eta}\theta_{\mu,\mu}=\tau_{\mu\eta,\mu\eta}\not=0$.
Hence $\theta_{\mu,\mu}\not=0$.
\end{proof}

In the proof of the next lemma we need some notation from
\cite{FMY2005}. Given subsets $U$ and $V$ of a finitely aligned
$P$-graph $\Lambda$, we write $\Ext(U;V)$ for the set
\[
\{\alpha \in \Lambda :\text{ there exist }\mu \in U\text{ and }\nu \in V\text{ such that }\mu\alpha \in \MCE(\mu,\nu)\}.
\]
Roughly speaking, $\Ext(U;V)$ is the set of tails which extend
paths in $U$ to minimal common extensions with paths in $V$.
Since $\Lambda$ is finitely aligned, if $U$ and $V$ are finite,
then so is $\Ext(U;V)$.

\begin{lemma}\label{lem:lower bound on norm}
Let $(G,P)$ be a quasi-lattice ordered group, and let $\Lambda$
be a finitely aligned $P$-graph. Let $\tau$ be a representation
of $\bal{\Lambda}$ which satisfies conditions \textup{(\ref{it:t lambda
nonzero})}~and~\textup{(\ref{it:gaps nonzero})} of Theorem~\textup{\ref{thm:Bb
uniqueness}}. Let $F$ be a finite $\vee$-closed subset of $P$
and let $m$ be a minimal element of $F$. For each $p \in F$,
let $X_p$ be a finite subset of $\Lambda^p$, and let $X =
\bigcup_{p \in F} X_p$. Fix scalars $\{a_{\mu,\nu} : (\mu,\nu)
\in \bal{X}\}$. Then
\[
\Big\| \sum_{(\mu,\nu) \in \bal{X}} a_{\mu,\nu} \tau_{\mu,\nu} \Big\|
    \ge \Big\| \sum_{(\mu,\nu) \in X_m *_s X_m} a_{\mu,\nu} \tau_{\mu,\nu} \Big\|.
\]
\end{lemma}
\begin{proof}
For each $p \in F$, let $a_p := \sum_{(\mu,\nu) \in X_p *_s
X_p} a_{\mu,\nu} \tau_{\mu,\nu}$, and let $a := \sum_{p \in F}
a_p$. We must prove that $\|a\|\ge \|a_m\|$. If $\|a_m\|=0$, then the
result is trivial. So assume $\|a_m\|>0$, and in particular
$X_m\not=\emptyset$. Define
\[
H:=\Ext(X_m;X\setminus X_m)\quad\text{and}\quad
Q:=\sum_{\mu\in X_m}\Big(\prod_{\lambda\in s(\mu)H}(\tau_{\mu,\mu}-\tau_{\mu\lambda,\mu\lambda})\Big).
\]
Then $Q$ is a projection, and with $\theta_{\mu,\nu}$ defined
as in Lemma~\ref{Lem: matrix units},
\[
a_mQ = \sum_{(\mu,\nu)\in X_m*_sX_m}a_{\mu,\nu}\tau_{\mu,\nu}Q
     = \sum_{(\mu,\nu)\in X_m*_sX_m}a_{\mu,\nu}\tau_{\mu,\nu}\theta_{\nu,\nu}
     = \sum_{(\mu,\nu)\in X_m*_sX_m}a_{\mu,\nu}\theta_{\mu,\nu}.
\]
We claim that $a_pQ=0$ for all $p\in F\setminus\{m\}$. To see
this, fix $p\in F\setminus\{m\}$ and $\rho,\sigma\in X_p$. If
$\MCE(\mu,\sigma)=\emptyset$ for all $\mu\in X_m$, then
$\tau_{\rho,\sigma}Q=0$. If there exists $\mu\in X_m$ with
$\MCE(\mu,\sigma)\not=\emptyset$, then there exists $\lambda\in
H$ with $\mu\lambda=\sigma$. So
$\tau_{\rho,\sigma}(\tau_{\mu,\mu}-\tau_{\mu\lambda,\mu\lambda})=0$,
and it follows that $\tau_{\rho,\sigma}Q=0$. This proves the claim.

It follows from (B1)~and~(B2) that
$\{\tau_{\mu,\nu}:(\mu,\nu)\in X_m *_s X_m\}$ is a family of
matrix units. They are nonzero because each
$\tau_{\mu,\mu}\not=0$. By Lemma~\ref{Lem: matrix units}, the
set $\{\theta_{\mu,\nu}:(\mu,\nu)\in X_m *_s X_m\}$ is also a
family of nonzero matrix units, so as in
Notation~\ref{ntn:compacts},
$\theta_{\mu,\nu}\mapsto\tau_{\mu,\nu}$ determines an
isomorphism $\clsp\{\theta_{\mu,\nu}:(\mu,\nu)\in
X_m*_sX_m\}\to \clsp\{\tau_{\mu,\nu}:(\mu,\nu)\in X_m*_sX_m\}$.
It follows that
\[
\|a\|\ge \|aQ\|
    =\|a_mQ\|
    =\Big\|\sum_{(\mu,\nu)\in X_m*X_m}a_{\mu,\nu}\theta_{\mu,\nu}\Big\|
    =\Big\| \sum_{(\mu,\nu) \in X_m *_s X_m} a_{\mu,\nu} \tau_{\mu,\nu} \Big\|,
\]
completing the proof.
\end{proof}

\begin{proof}[Proof of Theorem~\ref{thm:Bb uniqueness}]
The ``only if" statement follows from Proposition~\ref{prp:Bb
generators nonzero}. For the ``if" statement, suppose that
$\tau$ satisfies (\ref{it:t lambda nonzero})~and~(\ref{it:gaps
nonzero}).

Fix a nonempty finite $\vee$-closed subset $F$ of $P$, and let
$m$ be a minimal element of $F$. For each $p \in F$, fix $b_p
\in B_p$, and let $b = \sum_{p\in F}b_p$. Suppose $b \not= 0$;
we will show that $\rho_\tau(b) \not= 0$.

For each $p\in F$ and each finite subset $X\subseteq\Lambda^p$,
let $Q_X$ denote the projection $\sum_{\lambda\in
X}\omega_{\lambda,\lambda}$. By Lemma~\ref{lem:Bb
structure}(\ref{it:B_p}) and Notation~\ref{ntn:compacts}, we
have
\[
b_p=\lim_{X\subseteq\Lambda^p,|X|<\infty}Q_X b_p Q_X.
\]

For each $p\in F$ fix a finite set $X_p \subset \Lambda^p$ such
that
\[
\|b_p-Q_{X_p}b_pQ_{X_p}\|\le\frac{\|b_m\|}{4|F|}.
\]
In particular, $\|Q_{X_m}b_mQ_{X_m}\| \ge 3\|b_m\|/4$.

For each $p\in F$ we have $Q_{X_p}b_pQ_{X_p} \in
\lsp\{\omega_{\mu,\nu}:\mu,\nu\in\Lambda^p\}$. Lemma~\ref{lem:lower
bound on norm} therefore implies that $\|\rho_\tau(\sum_{p\in
F}Q_{X_p}b_pQ_{X_p})\|\ge\|Q_{X_m} b_m Q_{X_m}\|$. We now have
\begin{align*}
\|\rho_\tau(b)\|
    &= \Big\|\rho_\tau(b)-\rho_\tau\Big(\sum_{p\in F}Q_{X_p}b_pQ_{X_p}\Big)+\rho_\tau\Big(\sum_{p\in F}Q_{X_p}b_pQ_{X_p}\Big)\Big\|\\
    &\ge \Big\|\rho_\tau\Big(\sum_{p\in F}Q_{X_p}b_pQ_{X_p}\Big)\Big\|
        - \Big\|\rho_\tau(b)-\rho_\tau\Big(\sum_{p\in F}Q_{X_p}b_pQ_{X_p}\Big)\Big\|\\
    &\ge \Big\|\rho_\tau\Big(\sum_{p\in F}Q_{X_p}b_pQ_{X_p}\Big)\Big\|
        - \sum_{p\in F}\Big\|\rho_\tau(b_p)-\rho_\tau(Q_{X_p}b_pQ_{X_p})\Big\|\\
    &\ge \|Q_{X_m} b_m Q_{X_m}\|-\sum_{p\in F}\frac{\|b_m\|}{4|F|}\\
    &\ge \frac{3\|b_m\|}{4} - \frac{\|b_m\|}{4} \\
    &> 0.
\end{align*}
Hence $\rho_\tau(b)\not=0$, and it follows that $\rho_\tau$ is
injective on $B_F$. It now follows from Lemma~\ref{lem:Bb
structure}(\ref{it:Bb directlim}) that $\rho_\tau$ is injective
on $\bal{\Lambda}$; and Lemma~\ref{lem:Bb structure}(\ref{it:Bb
directlim})~and~(\ref{it:B_F closed}) show that $\bal{\Lambda}$
is AF because direct limits of AF algebras are also AF (see,
for example, \cite[Theorem~III.3.4]{Davidson1996}).
\end{proof}

\subsection{Ideals of the balanced algebra}\label{subsec:bal
ideals}

In this subsection we prove our key technical result,
Theorem~\ref{thm:ideals in core}. This theorem identifies generating elements for any ideal $I\triangleleft\BalAlg{\Lambda}$
which contains none of the generators $\omega_{\mu,\nu}$. We use Theorem~\ref{thm:ideals in core} in
the next subsection to see that there is a unique largest ideal containing no $\omega_{\mu,\nu}$. The corresponding quotient of $\BalAlg{\Lambda}$ is the desired
co-universal balanced algebra of $\Lambda$.

\begin{theorem}\label{thm:ideals in core}
Let $(G,P)$ be a quasi-lattice ordered group, and let $\Lambda$
be a finitely aligned $P$-graph. Let $\tau$ be a representation
of $\bal{\Lambda}$ such that each $\tau_{\mu,\mu}$ is nonzero,
and let $\rho_\tau : \BalAlg{\Lambda} \to C^*(\tau)$ be the
homomorphism induced by the universal property of
$\BalAlg{\Lambda}$. Then $\ker(\rho_\tau)$ is generated by the
set
\[
\bigcup_{\mu \in \Lambda}
    \Big\{ \prod_{\alpha \in E} (\omega_{\mu,\mu} - \omega_{\mu\alpha,\mu\alpha})
        : E \subset s(\mu)\Lambda\text{ is finite exhaustive and }
            \prod_{\alpha \in E} (\tau_{\mu,\mu} - \tau_{\mu\alpha,\mu\alpha}) = 0 \Big\}.
\]
\end{theorem}

By Lemma~\ref{lem:Bb structure}(\ref{it:Bb
directlim}), it suffices to show that the ideals of the $B_F$
are all generated by the appropriate elements. We require the following technical lemma.

\begin{lemma}\label{lem:B_F decomp}
Let $(G,P)$ be a quasi-lattice ordered group, and let $\Lambda$
be a finitely aligned $P$-graph. Let $F$ be a finite
$\vee$-closed subset of $P$, let $m$ be a minimal element of
$F$, and let $H$ be a finite subset of $\Lambda^m$. Then $A_H
:= \lsp\{\omega_{\mu,\nu} : (\mu,\nu) \in H *_s H\} + B_{F
\setminus \{m\}}$ is a $C^*$-subalgebra of $B_F$ and $B_F =
\varinjlim A_H$.
\end{lemma}
\begin{proof}
Calculations using (B1)~and~(B2) show that
$\lsp\{\omega_{\mu,\nu} : (\mu,\nu) \in H *_s H\}$ is a
finite-dimensional $C^*$-algebra. Using~(B2) and that $m$ is
minimal, one checks that $\lsp\{\omega_{\mu,\nu} : (\mu,\nu)
\in H *_s H\}$ is absorbed by $B_{F \setminus\{m\}}$ under
multiplication. The result then follows from
\cite[Corollary~1.8.4]{Dixmier1977} since each spanning element
of $B_F$ belongs to some $A_H$.
\end{proof}

\begin{prop}\label{prop:ideals in B F}
Let $(G,P)$ be a quasi-lattice ordered group, and let $\Lambda$
be a finitely aligned $P$-graph. Let $\tau$ be a representation
of $\bal{\Lambda}$ such that each $\tau_{\mu,\mu}$ is nonzero,
and let $\rho_\tau : \BalAlg{\Lambda} \to C^*(\tau)$ be the
homomorphism induced by the universal property of
$\BalAlg{\Lambda}$. Then for each finite $\vee$-closed subset
$F$ of $P$, $\ker(\rho_\tau)\cap B_F$ is generated by the set
\begin{equation}\label{eq:set gen ideal in B F}\begin{split}
\bigcup_{q \in F} \Big\{ \prod_{\alpha \in E} (\omega_{\mu,\mu} - \omega_{\mu\alpha,\mu\alpha})
        : \mu \in \Lambda^q, E \subset \bigcup_{p \in F, q \le p} s(\mu)&\Lambda^{q^{-1}p}{} \text{ is finite exhaustive},\\
            &\text{ and } \prod_{\alpha \in E} (\tau_{\mu,\mu} - \tau_{\mu\alpha,\mu\alpha}) = 0 \Big\}.
\end{split}
\end{equation}
\end{prop}

We will need the following notation in the proof of the
proposition. Given a finite subset $G$ of $\Lambda$, we define
$\MCE(G) := \big\{\lambda \in \bigcap_{\mu \in G} \mu\Lambda :
d(\lambda) = \bigvee_{\mu \in G} d(\mu)\big\}$; in particular, if $\bigvee_{\mu \in G} d(\mu)=\infty$, then $\MCE(G)=\emptyset$. Given a subset
$F$ of $\Lambda$, we define
\[
\vee F := \bigcup_{G \subset F} \MCE(G).
\]
Since $\Lambda$ is finitely aligned, $\vee F$ is finite. It is
not hard to see that $\vee F$ is the smallest subset of
$\Lambda$ which contains $F$ and is closed under taking minimal
common extensions.

\begin{proof}[Proof of Proposition~\ref{prop:ideals in B F}]
For each finite $\vee$-closed subset $F$ of $P$ denote by $I^F_\tau$ the ideal of $B_F$ generated by the
set~(\ref{eq:set gen ideal in B F}). Fix a finite $\vee$-closed set $F$. Since each element of the
set~(\ref{eq:set gen ideal in B F}) belongs to $\ker(\rho_\tau)
\cap B_F$ by definition, we have $I^F_\tau \subset
\ker(\rho_\tau) \cap B_F$, so it suffices to establish the
reverse inclusion.

We proceed by induction on $|F|$. Since $\rho_\tau$ is
injective on each $B_p$, we have $\{0\}=\ker(\rho_\tau) \cap
B_F \subset I^F_\tau$ if $|F| = 1$. Suppose now that
$\ker(\rho_\tau) \cap B_F \subset I^F_\tau$ whenever $|F| < k$,
and fix $F$ with $|F|=k$. Let $m$ be a minimal element of $F$,
and define $G:=F\setminus\{m\}$. Recall from Lemma~\ref{lem:B_F
decomp} that for finite $H\subseteq\Lambda^m$, $A_H$ denotes
$\lsp\{\omega_{\mu,\nu} : (\mu,\nu) \in H *_s H\} + B_{F
\setminus\{m\}}$. By Lemma~\ref{lem:B_F decomp}, it suffices to
show that $\ker(\rho_\tau)\cap A_H \subset I^F_\tau \cap A_H$
for all $H$.

Fix scalars $\{a_{\mu,\nu} : (\mu,\nu) \in H *_s H\}$. Without
loss of generality, we may assume that for each $v\in s(H)$
there exists $\mu,\nu\in Hv$ with $a_{\mu,\nu}\not=0$. Fix
$T_p\in B_p$ for each $p\in G$. Then
\[
a:=\sum_{(\mu,\nu) \in H *_s H}a_{\mu,\nu}\omega_{\mu,\nu}+\sum_{p\in G}T_p
\]
is a typical element of $A_H$. Suppose $\rho_\tau(a)=0$. We
must show that $a \in I^F_\tau$. We proceed in three steps:

\begin{tabular}{lp{0.8\textwidth}}%
    1. & Decompose $a$ as $a = a_{\{m\}} + a_G$ where $a_G \in B_G$, and $a_{\{m\}}$ is
    a linear combination of elements of the form
    $\prod_{\alpha\in E}(\omega_{\mu,\mu}-\omega_{\mu\alpha,\mu\alpha}) \omega_{\mu,\nu}$
    where $E$ is a finite exhaustive $\vee$-closed subset of
    $\bigcup_{p \in G, m \le p} \Lambda^{m^{-1}p}$ (see~\eqref{eq:a decomp} below). \\
    2. & Show that $a_{\{m\}} \in \ker(\rho_\tau)$. \\
    3. & Deduce that $a_G \in \ker(\rho_\tau)$ and then
    apply the inductive hypothesis.
\end{tabular}

\smallskip \noindent\textsc{Step 1.} (\emph{Decompose $a$ as $a = a_{\{m\}} + a_G$}.) Let
\[
M:=\min_{v\in s(H)}\Big\|\sum_{\mu,\nu\in Hv}a_{\mu,\nu}\omega_{\mu,\nu}\Big\|.
\]
Then $M\not=0$ by our assumption on the scalars $a_{\mu,\nu}$.
For each $p\in G$ the isomorphism $B_p \cong \bigoplus_{v \in
\Lambda^0}\Kk_{\Lambda^p v}$ of Lemma~\ref{lem:Bb
structure}(\ref{it:B_p}) and the approximate identities of the
$\Kk_{\Lambda^p v}$ obtained from Notation~\ref{ntn:compacts}
imply that there is a finite subset $F_p\subseteq\Lambda^p$
such that $\|T_p-(\sum_{\lambda\in
F_p}\omega_{\lambda,\lambda})T_p\|< M/(2|G|{|H|}^2)$. Without
loss of generality, we may assume that if $(\mu,\nu) \in H *_s
H$ and $\mu\alpha\in F_p$, then $\nu\alpha\in F_p$. Let
\[
E_p:=\{\alpha \in \Lambda^{m^{-1}p} : \text{ there exists } \mu \in\Lambda^m\text{ such that }\mu\alpha\in F_p\}.
\]
Let $E_G := \bigcup_{p \in G} E_p$. For each $v \in \Lambda^0$, the set $v(\vee E_G)$ is equal to $\vee(v E_G)$ and so is closed
under minimal common extensions. Each $v(\vee E_G)$ is
finite because $\Lambda$ is finitely aligned.

For each $\mu\in H$ we have
\[
\Big\|\Big(\omega_{\mu,\mu}-\sum_{\alpha\in s(\mu)E_p}\omega_{\mu\alpha,\mu\alpha}\Big)T_p\Big\|
    =\Big\|\omega_{\mu,\mu}\Big(T_p-\Big(\sum_{\lambda\in F_p}\omega_{\lambda,\lambda}\Big)T_p\Big)\Big\|<\frac{M}{2|G|{|H|}^2}.
\]
Hence
\begin{equation}\label{eq:cutting down T p}
\Big\|\sum_{\mu\in H}\Big(\omega_{\mu,\mu}-\sum_{\alpha\in s(\mu)E_p}\omega_{\mu\alpha,\mu\alpha}\Big)T_p\Big\|
    < \sum_{\mu\in H}\frac{M}{2|G|{|H|}^2}=\frac{M}{2|G||H|}.
\end{equation}

\smallskip\noindent\textbf{Claim.} Each $v(\vee E_G)$ is
exhaustive.

To prove this claim, we suppose that $v\in s(H)$ has the
property that $v(\vee E_G)$ is not exhaustive and seek a
contradiction. Since $v(\vee E_G)$ is not exhaustive, there
exists $\lambda\in v\Lambda$ with
$\MCE(\lambda,\alpha)=\emptyset$ for all $\alpha\in v(\vee
E_G)$. We have
\[
\Big(\sum_{\mu\in Hv}\omega_{\mu\lambda,\mu\lambda}\Big)a
    =\sum_{\mu,\nu\in Hv}a_{\mu,\nu}\omega_{\mu\lambda,\nu\lambda} + \sum_{\substack{p\in G \\ \mu\in Hv}}\omega_{\mu\lambda,\mu\lambda}T_p.
\]
Since $\rho_{\tau}(a)=0$, applying $\rho_{\tau}$ to both sides
of the above equation and rearranging yields
\[
\sum_{\mu,\nu\in Hv}a_{\mu,\nu}\tau_{\mu\lambda,\nu\lambda}=-\rho_\tau\Big(\sum_{\substack{p\in G \\ \mu\in Hv}}\omega_{\mu\lambda,\mu\lambda}T_p\Big).
\]
It follows that
\begin{align}
\Big\|\sum_{\mu,\nu\in Hv}a_{\mu,\nu}\tau_{\mu\lambda,\nu\lambda}\Big\|
    &\le \sum_{p\in G}\Big\|\sum_{\mu\in Hv}\omega_{\mu\lambda,\mu\lambda}T_p\Big\| \nonumber\\
    &= \sum_{p\in G}\Big\|\sum_{\mu\in Hv}\omega_{\mu\lambda,\mu\lambda}\omega_{\mu,\mu}T_p\Big\| \nonumber\\
    &\le \sum_{p\in G}\Big(\Big\|\sum_{\mu\in Hv}\omega_{\mu\lambda,\mu\lambda}
        \Big(\sum_{\alpha\in s(\mu)E_p}\omega_{\mu\alpha,\mu\alpha}\Big)T_p\Big\|\Big)\nonumber\\
    &\qquad+\sum_{p\in G}\Big(\Big\|\sum_{\mu\in Hv}\omega_{\mu\lambda,\mu\lambda}
        \Big(\omega_{\mu,\mu}-\sum_{\alpha\in s(\mu)E_p}\omega_{\mu\alpha,\mu\alpha}\Big)T_p\Big\|\Big),\label{eq:stopping point for big calc}
\end{align}
where the last inequality follows from the triangle inequality.
Since $\MCE(\lambda,\alpha)=\emptyset$ for all $\alpha\in v(\vee E_G)$, it follows from (B2) that
\[
\sum_{p\in G}\Big(\Big\|\sum_{\mu\in Hv}\omega_{\mu\lambda,\mu\lambda}
        \Big(\sum_{\alpha\in s(\mu)E_p}\omega_{\mu\alpha,\mu\alpha}\Big)T_p\Big\|\Big)=0.
\]
Using that $H \subset \Lambda^m$ in the second
line, we calculate
\begin{align*}
\Big\|\sum_{\mu,\nu\in Hv}a_{\mu,\nu}\tau_{\mu\lambda,\nu\lambda}\Big\|
    &\le  \sum_{p\in G}\Big(\Big\|\sum_{\mu\in Hv}\omega_{\mu\lambda,\mu\lambda}
        \Big(\omega_{\mu,\mu}-\sum_{\alpha\in s(\mu)E_p}\omega_{\mu\alpha,\mu\alpha}\Big)T_p\Big\|\Big)\nonumber\\
    &= \sum_{p\in G}\Big(\Big\|\sum_{\mu\in Hv}\omega_{\mu\lambda,\mu\lambda}
        \Big(\sum_{\mu'\in H}\Big(\omega_{\mu',\mu'}-\sum_{\alpha\in s(\mu')E_p}\omega_{\mu'\alpha,\mu'\alpha}\Big)\Big) T_p\Big\|\Big)\nonumber\\
    &\le \sum_{p\in G}\Big(\sum_{\mu\in Hv}\|\omega_{\mu\lambda,\mu\lambda}\|\,
        \Big\|\sum_{\mu'\in H}\Big(\omega_{\mu',\mu'} -\sum_{\alpha\in s(\mu')E_p}\omega_{\mu'\alpha,\mu'\alpha}\Big)T_p\Big\|\Big)\nonumber\\
    &= \sum_{\substack{p\in G \\ \mu\in Hv}}
        \Big\|\sum_{\mu'\in H}\Big(\omega_{\mu',\mu'} -\sum_{\alpha\in s(\mu')E_p} \omega_{\mu'\alpha,\mu'\alpha}\Big)T_p\Big\|.\nonumber
\end{align*}
Equation~(\ref{eq:cutting down T p}) therefore implies that
\begin{equation}\label{calc for exhaustive set}
\Big\|\sum_{\mu,\nu\in Hv}a_{\mu,\nu}\tau_{\mu\lambda,\nu\lambda}\Big\|
    <\sum_{\substack{p\in G \\ \mu\in Hv}}\frac{M}{2|G||H|}
    =\frac{|Hv|M}{2|H|}\le\frac{M}{2}.
\end{equation}
Since $\tau_{\mu,\mu}\not=0$ for all $\mu \in \Lambda$, we have
$\tau_{\mu\lambda,\nu\lambda}\not=0$ for all $\mu,\nu\in Hv$.
In particular, both $\{\omega_{\mu,\nu} : \mu,\nu \in Hv\}$ and
$\{\tau_{\mu\lambda, \nu\lambda} : \mu,\nu \in Hv\}$ are
families of nonzero matrix units. Hence the uniqueness of
$\Kk_{Hv}$ implies that $\omega_{\mu,\nu}\mapsto
\tau_{\mu\lambda,\nu\lambda}$ extends to an isomorphism
$C^*(\{\omega_{\mu,\nu} : \mu,\nu\in Hv\})\cong
C^*(\{\tau_{\mu\lambda,\nu\lambda} : \mu,\nu\in Hv\})$. Thus
\[
\Big\|\sum_{\mu,\nu\in Hv}a_{\mu,\nu}\tau_{\mu\lambda,\nu\lambda}\Big\|=\Big\|\sum_{\mu,\nu\in Hv}a_{\mu,\nu}\omega_{\mu,\nu}\Big\|\ge M.
\]
This contradicts (\ref{calc for exhaustive set}), completing
the proof of the claim.

To finish off Step~1, for each $v \in s(H)$, each $\mu \in Hv$,
and each $\alpha \in v(\vee E_G)$, let
\[
Q_{\mu\alpha}^{v(\vee E_G)} := \omega_{\mu\alpha,\mu\alpha} \prod_{\alpha\zeta \in v(\vee E_G)\setminus\{\alpha\}}
    (\omega_{\mu\alpha, \mu\alpha} - \omega_{\mu\alpha\zeta, \mu\alpha\zeta}).
\]
Each $Q^{v(\vee E_G)}_{\mu\alpha} \in B_G$ by definition of the
$v(\vee E_G)$. Let $K := \{v, \mu\} \cup \mu (\vee E_G)$. Then
$\vee K = K$. Reversing the edges in $\Lambda$ yields a
finitely aligned product system over $P$ of graphs as in
\cite[Example~3.1]{RS2005}, and then the arguments of
\cite[Proposition~8.6]{RS2005} applied to the set $K$ show that
\begin{equation}\label{Should have been in RaS}
\omega_{\mu,\mu}=\prod_{\alpha\in s(\mu) (\vee E_G)}(\omega_{\mu,\mu}-\omega_{\mu\alpha,\mu\alpha})
        +\sum_{\alpha\in s(\mu)(\vee E_G)}Q_{\mu\alpha}^{s(\mu) (\vee E_G)}
\end{equation}
(see the displayed equation immediately below equation~(8.5) on
page~421 of \cite{RS2005}).

Let
\begin{equation}\label{eq:a decomp}\begin{split}
a_{\{m\}} &:= \sum_{(\mu,\nu) \in H *_s H}a_{\mu,\nu}
    \Big(\prod_{\alpha\in s(\mu) (\vee E_G)}(\omega_{\mu,\mu}-\omega_{\mu\alpha,\mu\alpha}) \omega_{\mu,\nu}\Big)
        \quad\text{ and}\\
a_G &:= \sum_{(\mu,\nu) \in H *_s H}a_{\mu,\nu}
    \Big(\sum_{\alpha\in s(\mu) (\vee E_G)}Q_{\mu\alpha}^{s(\mu) (\vee E_G)}\omega_{\mu,\nu}\Big) + \sum_{p\in G}T_p.
\end{split}\end{equation}
It follows from \eqref{Should have been in RaS} that $a$
decomposes as $a = a_{\{m\}} + a_G$, which is the desired
decomposition of $a$. This completes Step~1.

\smallskip\noindent\textsc{Step~2.} (\emph{Show that $a_{\{m\}} \in
\ker(\rho_\tau)$}.)

To begin Step~2, fix $\mu'\in H$ and $\alpha'\in E_{s(\mu)}$.
We claim that
\begin{equation}\label{eq:orthogonal}
\Big(\sum_{\mu\in H}\prod_{\alpha\in s(\mu) (\vee E_G)}(\omega_{\mu,\mu}-\omega_{\mu\alpha,\mu\alpha})\Big)Q_{\mu'\alpha'}^{s(\mu') (\vee E_G)}=0.
\end{equation}
First suppose that $\mu \not= \mu'$. Then  $H \subset
\Lambda^m$ implies $\omega_{\mu,\mu}\omega_{\mu',\mu'}=0$,
giving~\eqref{eq:orthogonal} when $\mu \not= \mu'$. Now suppose
that $\mu = \mu'$. Then the product $\prod_{\alpha\in s(\mu)
(\vee E_G)}(\omega_{\mu,\mu}-\omega_{\mu\alpha,\mu\alpha})$
contains a factor of
$\omega_{\mu',\mu'}-\omega_{\mu'\alpha',\mu'\alpha'}$. Since
$\omega_{\mu',\mu'} \ge \omega_{\mu'\alpha',\mu'\alpha'} \ge
Q_{\mu'\alpha'}^{s(\mu) (\vee E_G)}$, we have
$(\omega_{\mu',\mu'}-\omega_{\mu'\alpha',\mu'\alpha'})
Q_{\mu'\alpha'}^{s(\mu') (\vee E_G)} = 0$, and this
establishes~\eqref{eq:orthogonal} when $\mu = \mu'$.

Using~\eqref{eq:a decomp} and then~\eqref{eq:orthogonal}, we
calculate:
\begin{align*}
\Big(\sum_{\mu\in H}\prod_{\alpha\in s(\mu) (\vee E_G)}&(\omega_{\mu,\mu}-\omega_{\mu\alpha,\mu\alpha})\Big)a \\
    &{}= \Big(\sum_{\mu\in H}\prod_{\alpha\in s(\mu) (\vee E_G)}(\omega_{\mu,\mu}-\omega_{\mu\alpha,\mu\alpha})\Big) (a_{\{m\}} + a_G) \\
    &{}= \sum_{(\mu,\nu) \in H *_s H}a_{\mu,\nu}\Big(\prod_{\alpha\in s(\mu) (\vee E_G)}(\omega_{\mu,\mu}-\omega_{\mu\alpha,\mu\alpha})\omega_{\mu,\nu}\Big) \\
        &\hskip8em{} + \sum_{\substack{p\in G \\ \mu\in H}}\Big(\prod_{\alpha\in s(\mu) (\vee E_G)}(\omega_{\mu,\mu}-\omega_{\mu\alpha,\mu\alpha})\Big)T_p.
\end{align*}
Since $\rho_{\tau}(a)=0$, applying $\rho_{\tau}$ to both sides
of the above equation and rearranging yields
\begin{equation}\label{eq:calc for ker elt}
\begin{split}
\sum_{(\mu,\nu) \in H *_s H}a_{\mu,\nu}\Big(\prod_{\alpha\in s(\mu) (\vee E_G)}{}&{}(\tau_{\mu,\mu}-\tau_{\mu\alpha,\mu\alpha})\tau_{\mu,\nu}\Big)\\
    &=-\rho_\tau\Big(\sum_{p\in G,\, \mu\in H}\Big(\prod_{\alpha\in s(\mu) (\vee E_G)}(\omega_{\mu,\mu}-\omega_{\mu\alpha,\mu\alpha})\Big)T_p\Big).
\end{split}
\end{equation}
Using (\ref{eq:calc for ker elt}), we calculate:
\begin{align*}
\Big\|\sum_{(\mu,\nu) \in H *_s H}a_{\mu,\nu}&\Big(\prod_{\alpha\in s(\mu) (\vee E_G)}(\tau_{\mu,\mu}-\tau_{\mu\alpha,\mu\alpha})\tau_{\mu,\nu}\Big)\Big\|\nonumber\\
    &\le \sum_{p\in G}\Big\|\sum_{\mu\in H}\prod_{\alpha\in s(\mu) (\vee E_G)}(\omega_{\mu,\mu}-\omega_{\mu\alpha,\mu\alpha})T_p\Big\|\nonumber\\
    &= \sum_{p\in G}\Big\|\sum_{\mu\in H}\prod_{\alpha'\in s(\mu) (\vee E_G)}(\omega_{\mu,\mu}-\omega_{\mu\alpha',\mu\alpha'})
            \prod_{\alpha\in s(\mu)E_p}(\omega_{\mu,\mu}-\omega_{\mu\alpha,\mu\alpha})T_p\Big\|
\end{align*}
since the $\omega_{\mu,\mu}-\omega_{\mu\alpha,\mu\alpha}$ are
projections and $s(\mu) E_p \subset s(\mu)(\vee E_G)$.
Since $H \subset \Lambda^m$, condition~(B2) yields
\begin{align}
\Big\|\sum_{(\mu,\nu) \in H *_s H}&a_{\mu,\nu}\Big(\prod_{\alpha\in s(\mu) (\vee E_G)}(\tau_{\mu,\mu}-\tau_{\mu\alpha,\mu\alpha})\tau_{\mu,\nu}\Big)\Big\|\nonumber\\
    &\le \sum_{p\in G}\Big\|\Big(\sum_{\mu\in H}\prod_{\alpha'\in s(\mu) (\vee E_G)}(\omega_{\mu,\mu}-\omega_{\mu\alpha',\mu\alpha'})\Big)
        \Big(\sum_{\mu\in H}\prod_{\alpha\in s(\mu)E_p}(\omega_{\mu,\mu}-\omega_{\mu\alpha,\mu\alpha})T_p\Big)\Big\| \nonumber\\
    &\le \sum_{p\in G}\Big\|\sum_{\mu\in H}\prod_{\alpha\in s(\mu)E_p}(\omega_{\mu,\mu}-\omega_{\mu\alpha,\mu\alpha})T_p\Big\|\label{eq:gap sum bound}
\end{align}
since the $\prod_{\alpha\in s(\mu)(\vee E_G)}
(\omega_{\mu,\mu}-\omega_{\mu\alpha,\mu\alpha})$ are mutually
orthogonal so that their sum is a projection. The
$\omega_{\mu\alpha,\mu\alpha}$ where $\alpha \in s(\mu)E_p$ are
mutually orthogonal, so each $\prod_{\alpha\in
s(\mu)E_p}(\omega_{\mu,\mu}-\omega_{\mu\alpha,\mu\alpha}) =
\omega_{\mu,\mu}-\sum_{\alpha\in
s(\mu)E_p}\omega_{\mu\alpha,\mu\alpha}$. Hence combining
\eqref{eq:gap sum bound}~with~\eqref{eq:cutting down T p}, we
obtain
\begin{equation}\label{calc for showing 1st term in
ker}\begin{split}
\Big\|\sum_{(\mu,\nu) \in H *_s H}a_{\mu,\nu}\Big(\prod_{\alpha\in s(\mu) (\vee E_G)}&(\tau_{\mu,\mu}-\tau_{\mu\alpha,\mu\alpha})\tau_{\mu,\nu}\Big)\Big\| \\
    &\le \sum_{p\in G}\Big\|\sum_{\mu\in H}\Big(\omega_{\mu,\mu}-\sum_{\alpha\in s(\mu)E_p}\omega_{\mu\alpha,\mu\alpha}\Big)T_p\Big\|
    < \frac{M}{2}.
\end{split}\end{equation}

Now, we claim that $\prod_{\alpha\in v(\vee E_G)}
(\tau_{\mu,\mu}-\tau_{\mu\alpha,\mu\alpha}) \tau_{\mu,\nu}=0$
for all $v\in s(H)$ and $\mu,\nu\in Hv$. Suppose for
contradiction that $v\in s(H)$ and $\mu,\nu\in Hv$ with
$\prod_{\alpha\in v(\vee E_G)}
(\tau_{\mu,\mu}-\tau_{\mu\alpha,\mu\alpha})
\tau_{\mu,\nu}\not=0$. Then Lemma~\ref{Lem: matrix units}
and~\eqref{eq: pull tau through} imply that the set
$\{\prod_{\alpha\in v(\vee E_G)}
(\tau_{\mu,\mu}-\tau_{\mu\alpha,\mu\alpha}) \tau_{\mu,\nu} :
\mu,\nu\in Hv\}$ is a family of matrix units. Since
$\prod_{\alpha\in v(\vee E_G)}
(\tau_{\mu,\mu}-\tau_{\mu\alpha,\mu\alpha})
\tau_{\mu,\nu}\not=0$, all the matrix units are nonzero. It
follows that $\omega_{\mu,\nu}\mapsto\prod_{\alpha\in v(\vee
E_G)} (\tau_{\mu,\mu}-\tau_{\mu\alpha,\mu\alpha})
\tau_{\mu,\nu}$ determines an isomorphism
$C^*(\{\omega_{\mu,\nu}:\mu,\nu\in Hv\})\cong \Kk_{Hv}$. We
then have
\begin{align*}
\Big\|\sum_{(\mu,\nu) \in H *_s H}a_{\mu,\nu}\Big(\prod_{\alpha\in s(\mu)(\vee E_G)}&(\tau_{\mu,\mu}-\tau_{\mu\alpha,\mu\alpha})\tau_{\mu,\nu}\Big)\Big\|\\
&=\max_{w\in s(H)}\Big\|\sum_{\mu,\nu\in Hw}a_{\mu,\nu}\Big(\prod_{\alpha\in w(\vee E_G)}(\tau_{\mu,\mu}-\tau_{\mu\alpha,\mu\alpha})\tau_{\mu,\nu}\Big)\Big\|\\
&\ge\Big\|\sum_{\mu,\nu\in Hv}a_{\mu,\nu}\Big(\prod_{\alpha\in v(\vee E_G)}(\tau_{\mu,\mu}-\tau_{\mu\alpha,\mu\alpha})\tau_{\mu,\nu}\Big)\Big\|\\
&=\Big\|\sum_{\mu,\nu\in Hv}a_{\mu,\nu}\omega_{\mu,\nu}\Big\|\\
&\ge M,
\end{align*}
which contradicts (\ref{calc for showing 1st term in ker}). This establishes the claim, and hence
\begin{align*}
\rho_\tau\Big(\sum_{(\mu,\nu) \in H *_s H} a_{\mu,\nu}{}&{}\Big(\prod_{\alpha\in s(\mu)(\vee E_G)}
        (\omega_{\mu,\mu}-\omega_{\mu\alpha,\mu\alpha})\omega_{\mu,\nu}\Big)\Big) \\
    &= \sum_{(\mu,\nu) \in H *_s H}a_{\mu,\nu}\Big(\prod_{\alpha\in s(\mu)(\vee E_G)}
        (\tau_{\mu,\mu}-\tau_{\mu\alpha,\mu\alpha})\tau_{\mu,\nu}\Big)=0.
\end{align*}
This completes Step~2.

\smallskip\noindent\textsc{Step~3.} (\emph{Deduce that $a_G \in
\ker(\rho_\tau)$ and then apply the inductive hypothesis}.)

To complete the proof, observe that $a \in \ker(\rho_\tau)$ by
hypothesis. Step~1 gives $a_G = a - a_{\{m\}}$, and Step~2
gives $a_{\{m\}} \in \ker(\rho_\tau)$, and it follows that $a_G
\in \ker(\rho_\tau)$ as well. The inductive hypothesis now
forces $a_G \in I^G_\tau \subset I^F_\tau$. It therefore
suffices to show that $a_{\{m\}} \in I^F_\tau$. By definition
of $a_{\{m\}}$, it suffices to show that
\begin{equation}\label{eq:gaps zero}
\prod_{\alpha \in s(\lambda) \vee E_G} \tau_{\lambda,\lambda} - \tau_{\lambda\alpha,\lambda\alpha} = 0
    \qquad\text{  for every $\lambda \in H$.}
\end{equation}
So fix $\lambda \in
H$. Recall that by choice of $H$ there exist $\mu,\nu \in
Hs(\lambda)$ such that $a_{\mu,\nu} \not= 0$. Hence,
\begin{align*}
\prod_{\alpha \in s(\lambda) \vee E_G}{}&{} \tau_{\lambda,\lambda} - \tau_{\lambda\alpha,\lambda\alpha}\\
    &= \frac{1}{a_{\mu,\nu} }
        \tau_{\lambda, \mu} \Big(\prod_{\alpha \in s(\lambda) \vee E_G} \tau_{\mu,\mu} - \tau_{\mu\alpha,\mu\alpha}\Big)
        \rho_\tau(a_{\{m\}})
        \Big(\prod_{\alpha \in s(\lambda) \vee E_G} \tau_{\nu,\nu} - \tau_{\nu\alpha,\nu\alpha}\Big) \tau_{\nu, \lambda}.
\end{align*}
Since Step~2 forces $\rho_\tau(a_{\{m\}}) = 0$, this
establishes~\eqref{eq:gaps zero} as required.
\end{proof}

\subsection{The co-universal balanced
algebra}\label{subsec:couniv bal}

In this subsection we use the analysis of
Subsection~\ref{subsec:bal ideals} to establish that there is a
representation of $\bal{\Lambda}$ by nonzero partial isometries
which is co-universal in the sense that it factors through
every other representation of $\bal{\Lambda}$ by nonzero
partial isometries.

\begin{theorem}\label{thm:Bmin couniversal}
Let $(G,P)$ be a quasi-lattice ordered group, and let
$(\Lambda,d)$ be a finitely aligned $P$-graph. There is a
$C^*$-algebra $\Bmin{\Lambda}$ generated by a representation
$\Omega$ of $\bal{\Lambda}$ such that:
\begin{enumerate}
\item\label{it:Omegas nonzero} each $\Omega_{\mu,\nu}$ is
    nonzero, and
\item\label{it:Omegas couniversal} given any other
    representation $\tau$ of $\bal{\Lambda}$ with each
    $\tau_{\mu,\nu}$ nonzero, there is a homomorphism
    $\phi_\tau : C^*(\tau) \to \Bmin{\Lambda}$ satisfying
    $\phi_\tau(\tau_{\mu,\nu}) = \Omega_{\mu,\nu}$ for each
    $(\mu,\nu) \in \bal{\Lambda}$.
\end{enumerate}
Moreover, $\Bmin{\Lambda}$ is unique up to canonical
isomorphism, and given a representation $\tau$ of
$\bal{\Lambda}$ as in~\textup{(\ref{it:Omegas couniversal})}, the
homomorphism $\phi_\tau$ is injective if and only if
\[
\prod_{\alpha \in E} (\tau_{\mu,\mu} - \tau_{\mu\alpha,\mu\alpha}) = 0
    \quad\text{ for every $\mu \in \Lambda$ and finite exhaustive subset $E \subset s(\mu)\Lambda$.}
\]
\end{theorem}

\begin{proof}[Proof of Theorem~\ref{thm:Bmin couniversal}]
Recall from Section~\ref{sec:filters} that $\widehat{\Lambda}$
denotes the space of filters of $\Lambda$, and
$\widehat{\Lambda}_\infty$ denotes the space of ultrafilters.
Lemma~\ref{action}(\ref{it:action-preserves boundary}) implies
that the subspace $\ell^2(\widehat{\Lambda}_\infty) \subset
\ell^2(\widehat{\Lambda})$ is invariant for the partial
isometries $\{\Sfilters_\lambda : \lambda \in \Lambda\}$.
Define $\Omega : \bal{\Lambda} \to
\Bb(\ell^2(\widehat{\Lambda}_\infty))$ by $\Omega_{\mu,\nu} :=
(T_\mu T^*_\nu)|_{\ell^2(\widehat{\Lambda}_\infty)}$ for all
$(\mu,\nu) \in \bal{\Lambda}$. Since the restriction map from
$\Bb(\ell^2(\widehat{\Lambda}))$ to
$\Bb(\ell^2(\widehat{\Lambda}_\infty))$ is a homomorphism on
$C^*(\{T_{\lambda}:\lambda\in\Lambda\})$, Lemma~\ref{lem:filter
repn} implies that $\Omega$ is a representation of
$\bal{\Lambda}$. Fix $(\mu,\nu) \in \bal{\Lambda}$. By
Lemma~\ref{lem:plenty of ultrafilters}, there is an ultrafilter
$U$ of $\Lambda$ with $\nu \in U$. Hence $\Omega_{\mu,\nu} e_U
= e_{\mu\cdot(\nu^*\cdot U)} \not = 0$. Thus $\Omega$
satisfies~(\ref{it:Omegas nonzero}).

Fix a representation $\tau$ of $\bal{\Lambda}$ as
in~(\ref{it:Omegas couniversal}). The universal property of
$\BalAlg{\Lambda}$ yields homomorphisms $\rho_\tau :
\BalAlg{\Lambda} \to C^*(\tau)$ and $\rho_\Omega :
\BalAlg{\Lambda} \to \Bmin{\Lambda}$ such that
$\rho_\tau(\omega_{\mu,\nu}) = \tau_{\mu,\nu}$ and
$\rho_\Omega(\omega_{\mu,\nu}) = \Omega_{\mu,\nu}$. We will show
that $\ker(\rho_\tau) \subset \ker(\rho_\Omega)$;
condition~(\ref{it:Omegas couniversal}) will follow because
there is then a well defined homomorphism $\phi_\tau :
C^*(\tau) \to \Bmin{\Lambda}$ satisfying $\phi_\tau \circ
\rho_\tau = \phi_\Omega$.

By Theorem~\ref{thm:ideals in core}, it suffices to show that
whenever $\mu \in \Lambda$ and $E$ is a finite exhaustive
subset of $s(\mu)\Lambda$ such that $\prod_{\alpha \in E}
(\tau_{\mu,\mu} - \tau_{\mu\alpha,\mu\alpha}) = 0$, we have
$\prod_{\alpha \in E} (\Omega_{\mu,\mu} -
\Omega_{\mu\alpha,\mu\alpha}) = 0$. In particular, it is enough
to establish that
\begin{equation}\label{eq:all gaps zero}
\prod_{\alpha \in E} (\Omega_{\mu,\mu} - \Omega_{\mu\alpha,\mu\alpha}) = 0\quad\text{ for every $\mu \in \Lambda$
and finite exhaustive $E \subset s(\mu)\Lambda$};
\end{equation}
and this follows from Lemma~\ref{lem:ultrafilters-FEsets}.

We have now proved that $\Omega$ satisfies (\ref{it:Omegas
nonzero})~and~(\ref{it:Omegas couniversal}). To see that
$\Bmin{\Lambda}$ is unique up to canonical isomorphism, suppose
that $\tau$ is another representation of $\bal{\Lambda}$ such
that $C^*(\tau)$ satisfies (\ref{it:Omegas
nonzero})~and~(\ref{it:Omegas couniversal}). Then
property~(\ref{it:Omegas couniversal}) of $C^*(\tau)$ gives a
homomorphism from $\Bmin{\Lambda}$ to $C^*(\tau)$ which is an
inverse for $\phi_\tau$.

Finally, to see that $\phi_\tau$ is injective if and only if
$\prod_{\alpha \in E} (\tau_{\mu,\mu} -
\tau_{\mu\alpha,\mu\alpha}) = 0$ for every $\mu \in \Lambda$ and every
finite exhaustive subset $E \subset s(\mu)\Lambda$, observe
first that the only if implication follows from~\eqref{eq:all
gaps zero}. For the if implication, suppose that $\tau$ is a
representation of $\bal{\Lambda}$ with each $\tau_{\mu,\nu}$
nonzero, and with $\prod_{\alpha \in E} (\tau_{\mu,\mu} -
\tau_{\mu\alpha,\mu\alpha}) = 0$ for every $\mu \in \Lambda$
and finite exhaustive subset $E \subset s(\mu)\Lambda$. Then in
particular, $\ker(\rho_\tau)$ contains $\prod_{\alpha \in E}
(\omega_{\mu,\mu} - \omega_{\mu\alpha,\mu\alpha})$ for every
$\mu \in \Lambda$ and finite exhaustive subset $E \subset
s(\mu)\Lambda$, and it follows from Theorem~\ref{thm:ideals in
core} that $\ker(\rho_\Omega) \subset \ker(\rho_\tau)$. Since
$\rho_\Omega =  \phi_\tau \circ \rho_\tau$, it follows that
$\phi_\tau$ is injective.
\end{proof}

\section{The $C^*$-algebra of a $P$-graph}\label{sec:Pgraph
alg}

We are now ready to state and prove our main theorem. We define
what we mean by a representation of a $P$-graph, and we show
that the $C^*$-algebra $\Tt C^*(\Lambda)$ which is universal
for such representations admits a co-universal quotient. That
is, there is a smallest quotient of $\Tt C^*(\Lambda)$ in which
the canonical coaction of $G$ is preserved and the images of
all the generators are nonzero. The point is that the hard work
is largely already done in the results of
Section~\ref{sec:C0(G0)}: we use the uniqueness theorem for
$\BalAlg{\Lambda}$ established in Subsection~\ref{subsec:univ
bal} to identify $\BalAlg{\Lambda}$ with the fixed-point
algebra of $\Tt C^*(\Lambda)$, and we then present a fairly
generic argument, based on coaction theory, to bootstrap the
co-universal property of $\Bmin{\Lambda}$ up to the desired
co-universal quotient of $\Tt C^*(\Lambda)$.

\begin{dfn}\label{repn of Lambda}
Let $(G,P)$ be a quasi-lattice ordered group, and let $\Lambda$
be a finitely aligned $P$-graph. A \emph{representation} of
$\Lambda$ in a $C^*$-algebra $B$ is a map $t : \Lambda \to B$,
$\lambda \mapsto t_\lambda$ such that:
\begin{itemize}
\item[(T1)] $\{t_v : v \in \Lambda^0\}$ is a collection of
    mutually orthogonal projections;
\item[(T2)] $t_\mu t_\nu = t_{\mu\nu}$ whenever $s(\mu) =
    r(\nu)$;
\item[(T3)] $t^*_\mu t_\mu = t_{s(\mu)}$ for all $\mu \in
    \Lambda$; and
\item[(T4)] $t_\mu t^*_\mu t_\nu t^*_\nu = \sum_{\lambda
    \in \MCE(\mu,\nu)} t_\lambda t^*_\lambda$ for all
    $\mu,\nu \in \Lambda$.
\end{itemize}
\end{dfn}
As in \cite[Theorem~6.3]{Fowler2002}, there exists a
$C^*$-algebra $\Tt C^*(\Lambda)$ generated by a representation
$\Suniv$ of $\Lambda$ which is universal in the sense that
given any other representation $t$ of $\Lambda$, there is a
homomorphism $\pi_t : \Tt C^*(\Lambda) \to C^*(t) :=
C^*(\{t_\lambda : \lambda \in \Lambda\})$ satisfying $\pi_t
\circ \Suniv = t$.

We need to dip a little into the theory of coactions; but not
too far because the coactions we deal with are all coactions of
discrete groups. For more detail on coactions, see
\cite[Section~A.3]{EKQR2006}. The following summary is adapted
from \cite[Section~3]{pp_CLSV2009}. All tensor products of
$C^*$-algebras (here and later in the section) are minimal
tensor products.

Let $G$ be a discrete group, and let $g\mapsto U_g$ be the
universal unitary representation of $G$ in $C^*(G)$. There is a
homomorphism $\delta_G \colon C^*(G)\to C^*(G) \otimes C^*(G)$
determined by $\delta_G(U_g)=U_g\otimes U_g$. A full coaction
of $G$ on a $C^*$-algebra $A$ is an injective nondegenerate
homomorphism $\delta \colon A\to A\otimes C^*(G)$ such that
$(\delta\otimes \id_{C^*(G)})\circ \delta= (\id_A \otimes
\delta_G)\circ \delta$. The fixed-point algebra for $G$ is the
subalgebra $A_e^\delta := \{a \in A : \delta(a) = a \otimes
1_{C^*(G)}\}$. By \cite[Lemma~1.3(a)]{Quigg1996} there is a
conditional expectation $\Phi^\delta$ from $A$ to $A^\delta_e$
determined by $\Phi^\delta(a) = a$ if $a \in A^\delta_e$, and
$\Phi^\delta(a) = 0$ if $\delta(a) = a \otimes U_g$ for some
other $g \in G$. A \emph{normal} coaction is one for which
$\Phi^\delta$ is faithful on positive elements (there are a
number of equivalent characterisations of normality for
coactions, but this is the one most useful from our point of
view). Given a coaction $\delta : A \to A \otimes C^*(G)$,
there is a quotient $A^r$ of $A$, and a normal coaction
$\delta^n : A^r \to A^r \otimes C^*(G)$ such that the coaction
crossed-products $A \otimes_\delta G$ and $A^r
\otimes_{\delta^n} G$ are identical. The cosystem $(A^r,
\delta^n, G)$ is called the \emph{normalisation} of
$(A,G,\delta)$ (see \cite[Section~A.7]{EKQR2006}). We write
$q_\delta$ for the quotient map from $A$ to $A^r$. We have
$\delta^n \circ q_\delta = (q_\delta \otimes 1) \circ \delta$,
and $q_\delta$ restricts to an isomorphism $A_e^\delta\cong
(A^r)_e^{\delta^n}$ of fixed-point algebras. Indeed $q_\delta$
is isometric on $\clsp\{a \in A : \delta(a) = a \otimes U_g\}$
for each fixed $g \in G$.

A standard argument using the universal property of $\Tt
C^*(\Lambda)$ shows that there is a coaction $\delta$ of $G$ on
$\Tt C^*(\Lambda)$ which satisfies $\delta(\Suniv_\lambda) =
\Suniv_\lambda \otimes U_{d(\lambda)}$ for all $\lambda \in
\Lambda$.

\begin{rmk}\label{rmk:spanning}
Let $(G,P)$ be a quasi-lattice ordered group, let $\Lambda$ be
a finitely aligned $P$-graph, and let $t : \Lambda \to B$ be a
representation of $\Lambda$. Relation~(T4) and the
factorisation property imply that $t^*_\mu t_\nu =
\sum_{\mu\alpha = \nu\beta \in \MCE(\mu,\nu)} t_\alpha
t^*_\beta$, and it follows that $C^*(t) = \clsp\{t_\mu t^*_\nu
: \mu,\nu \in \Lambda\}$. Thus if $C^*(t)$ admits a coaction
$\alpha$ of $G$ satisfying $\alpha(t_\lambda) = t_\lambda
\otimes U_{d(\lambda)}$ for all $\lambda \in \Lambda$, then
that $C^*(t)^\alpha = \Phi^\alpha(C^*(t))$ forces
$C^*(t)^\alpha = \clsp\{t_\mu t^*_\nu : d(\mu) = d(\nu)\}$.
\end{rmk}

\begin{theorem}\label{thm:couniversal alg of Lambda}
Let $(G,P)$ be a quasi-lattice ordered group, and let $\Lambda$
be a finitely aligned $P$-graph. There exists a $C^*$-algebra
$\Csmin{\Lambda}$ generated by a representation $\Smin$ of
$\Lambda$ such that
\begin{enumerate}
\item\label{it:Smin nonzero} each $\Smin_\lambda$ is
    nonzero, and there is a coaction $\beta$ of $G$ on
    $\Csmin{\Lambda}$ satisfying $\beta(\Smin_\lambda) =
    \Smin_\lambda \otimes U_{d(\lambda)}$ for all
    $\lambda$; and
\item\label{it:Smin couniversal} given any other representation $t$ of $\Lambda$ with each $t_{\lambda}$ nonzero such that $C^*(t)$ carries a
    coaction $\alpha$ of $G$ satisfying $\alpha(t_\lambda)
    = t_\lambda \otimes U_{d(\lambda)}$ for all $\lambda$,
    there is a homomorphism $\psi_t : C^*(t) \to \Csmin{\Lambda}$
    satisfying $\psi_t(t_\lambda) = \Smin_\lambda$ for all
    $\lambda$.
\end{enumerate}
Moreover, $\Csmin{\Lambda}$ is unique up to canonical
isomorphism, and the homomorphism $\psi_t$
of~\textnormal{(\ref{it:Smin couniversal})} is injective if and
only if $\alpha$ is normal and
\[\textstyle
\prod_{\alpha \in E} (t_\mu t^*_\mu - t_{\mu\alpha} t^*_{\mu\alpha}) = 0
    \quad\text{ for every $\mu \in \Lambda$ and every finite exhaustive $E \subset s(\mu)\Lambda$}.
\]
\end{theorem}

To prove the theorem we require a preliminary result

\begin{lemma}
Let $(G,P)$ be a quasi-lattice ordered group, and let $\Lambda$
be a finitely aligned $P$-graph. Then there is an isomorphism
$\BalAlg{\Lambda} \cong \Tt C^*(\Lambda)^\delta$ which takes
$\omega_{\mu,\nu}$ to $\Suniv_\mu \Suniv^*_\nu$ for all
$(\mu,\nu) \in \bal{\Lambda}$.
\end{lemma}
\begin{proof}
It is routine that $\tau_{\mu,\nu} := \Suniv_\mu \Suniv^*_\nu$
determines a representation $\tau$ of $\bal{\Lambda}$. So there
is a homomorphism $\rho_\tau$ from $\BalAlg{\Lambda}$ to $\Tt
C^*(\Lambda)$ which takes each $\omega_{\mu,\nu}$ to
$\Suniv_\mu \Suniv^*_\nu$. The partial isometries
$\Sfilters_\lambda$ of Definition~\ref{dfn:Sfilters} clearly
satisfy $\Sfilters_{\mu}\Sfilters_{\nu} = \Sfilters_{\mu\nu}$
when $s(\mu) = r(\nu)$, and combined with Lemma~\ref{lem:filter
repn}, this shows that $\Sfilters$ is a representation of
$\Lambda$. Hence there is a homomorphism $\pi_{\Sfilters} : \Tt
C^*(\Lambda) \to C^*(\Sfilters)$ which takes each $\Suniv_\mu$
to $\Sfilters_\mu$. Thus \eqref{eq:balalg gens
nonzero}~and~\eqref{eq:balalg gaps nonzero} imply that each
$\Suniv_\mu \Suniv^*_\nu \not= 0$ and that for every $\mu \in
\Lambda$ and finite exhaustive set $E \subset s(\mu)\Lambda$,
\[\textstyle
\prod_{\alpha\in E}(\Suniv_\mu \Suniv^*_\mu - \Suniv_{\mu\alpha}\Suniv^*_{\mu\alpha}) \not= 0.
\]
Hence Theorem~\ref{thm:Bb uniqueness} implies that $\rho_\tau$
is injective.

It remains to show that the range of $\rho_\tau$ is $\Tt
C^*(\Lambda)^\delta$. We have $\Suniv_\mu \Suniv^*_\nu = 0$
unless $s(\mu) = s(\nu)$ by~(T3). Hence
\[
\range(\rho_\tau) = \clsp\{\Suniv_\mu \Suniv^*_\nu : d(\mu) = d(\nu)\},
\]
and this is equal to $\Tt C^*(\Lambda)^\delta$ by
Remark~\ref{rmk:spanning}.
\end{proof}

\begin{proof}[Proof of Theorem~\ref{thm:couniversal alg of
Lambda}] Lemma~\ref{action}(\ref{it:action-preserves boundary})
implies that $\ell^2(\widehat{\Lambda}_\infty) \subset
\ell^2(\widehat{\Lambda})$ is invariant for the partial
isometries $\Sfilters_\lambda$ of
Definition~\ref{dfn:Sfilters}. Hence the partial isometries
$\Sfilters_\lambda|_{\ell^2(\widehat{\Lambda}_\infty)}$ form a
representation of $\Lambda$ on
$\Bb(\ell^2(\widehat{\Lambda}_\infty))$.

Define a map $\tT : \Lambda \to
\Bb(\ell^2(\widehat{\Lambda}_\infty)) \otimes C^*(G)$ by
\[
\tT_\lambda := \Sfilters_\lambda|_{\ell^2(\widehat{\Lambda}_\infty)} \otimes U_{d(\lambda)}.
\]
It is straightforward to see that $\tT$ is a representation of
$\Lambda$.

Let $\beta_0$ be the canonical coaction on
$\Bb(\ell^2(\widehat{E}_\infty)) \otimes C^*(G)$ given by
$\beta_0(a \otimes U_g) := (a \otimes U_g) \otimes U_g$. Since
sums of the form $\sum_{v \in F} \tT_v$ where $F$ increases
over finite subsets of $\Lambda^0$ form an approximate identity
for $\Bb(\ell^2(\widehat{E}_\infty))$, and since $G$ is
discrete, $\beta_0$ restricts to a coaction, also denoted
$\beta_0$, on $C^*(\tT) = \clsp\{\tT_\mu \tT^*_\nu : \mu,\nu
\in \Lambda\}$ (see \cite[Remark~A.22(3)]{EKQR2006}). Let
$(\Csmin{\Lambda}, \beta)$ be the normalisation of the cosystem
$(C^*(\tT), \beta_0)$ as in \cite[Definition~A.56]{EKQR2006}:
so $\Csmin{\Lambda} = C^*(\tT)^r$ and $\beta = \beta_0^n$.
Recall that $q_{\beta_0}$ denotes the canonical quotient map
from $C^*(\tT)$ to $\Csmin{\Lambda}$. For each $\lambda \in
\Lambda$, let $\Smin_\lambda:=q_{\beta_0}(\tT_\lambda)$, so
$\Smin$ is a representation of $\Lambda$. Since $q_{\beta_0}$
is isometric on $\clsp\{a \in C^*(\tT) : \beta_0(a) = a \otimes
U_{d(\lambda)}\}$, it follows from Lemma~\ref{lem:plenty of
ultrafilters} that each $S_\lambda$ is nonzero. Hence the
triple $(\Csmin{\Lambda}, \Smin, \beta)$
satisfies~(\ref{it:Smin nonzero}).

To prove~(\ref{it:Smin couniversal}), fix a representation $t$ of $\Lambda$ and a coaction $\alpha$ of $G$ on
$C^*(t)$ as in~(\ref{it:Smin couniversal}). Let $\pi_t : \Tt
C^*(\Lambda) \to C^*(t)$ be the homomorphism obtained from the
universal property of $\Tt C^*(\Lambda)$. It suffices to show
that $\ker(\pi_t) \subset \ker(\pi_{\Smin})$ for then
$\pi_{\Smin}$ descends to the desired homomorphism from
$C^*(t)$ to $C^*_{\min}(\Lambda)$.

Let
\[
\Phi^\delta : \Tt C^*(\Lambda) \to \Tt C^*(\Lambda)^\delta,
\quad
\Phi^\alpha : C^*(t) \to C^*(t)^\alpha
\quad\text{ and } \quad
\Phi^\beta : \Csmin{\Lambda} \to \Csmin{\Lambda}^\beta
\]
be the conditional expectations obtained from the coactions
$\delta, \alpha, \beta$.

Fix $a \in \Tt C^*(\Lambda)$ with $\pi_t(a) = 0$. Then
$\Phi^\alpha(\pi_t(a^*a)) = 0$. Since $\pi_t$ intertwines
$\delta$ and $\alpha$ on spanning elements, we have
$\Phi^\alpha \circ \pi_t = \pi_t \circ \Phi^{\delta}$, and
hence
\begin{equation}\label{eq:composite image zero}
\pi_t(\Phi^{\delta}(a^*a)) = 0.
\end{equation}

\textbf{Claim.} There is an isomorphism $\Csmin{\Lambda}^\beta
\cong \Bmin{\Lambda}$ satisfying $\Smin_\mu \Smin^*_\nu \mapsto
\Omega_{\mu,\nu}$ for all $(\mu,\nu) \in \bal{\Lambda}$.

To prove the claim, first note that Remark~\ref{rmk:spanning}
implies that $\Csmin{\Lambda}^\beta$ is generated by the
representation of $\bal{\Lambda}$ defined by $(\mu,\nu) \mapsto
\Smin_\mu \Smin^*_\nu$. Since $q_{\beta_0}$ is injective on
$C^*(\tT)^{\beta_0}$, we have $\Smin_\mu \Smin^*_\nu \not= 0$
for all $(\mu,\nu) \in \bal{\Lambda}$. The co-universal
property of $\Bmin{\Lambda}$ therefore induces a homomorphism
$\phi_{\Smin} : \Csmin{\Lambda}^\beta \to \Bmin{\Lambda}$
satisfying $\phi_{\Smin}(\Smin_\mu \Smin^*_\nu) =
\Omega_{\mu,\nu}$. Lemma~\ref{lem:ultrafilters-FEsets} implies
that for each $\mu \in \Lambda$ and each finite exhaustive
subset $E$ of $s(\mu)\Lambda$,
\[\textstyle
\prod_{\alpha \in E} (\Smin_\mu \Smin^*_\mu - \Smin_{\mu\alpha}\Smin^*_{\mu\alpha}) = 0.
\]
Hence the final assertion of Theorem~\ref{thm:Bmin couniversal}
implies that $\phi_S$ is an isomorphism. This proves the claim.

The claim combined with the co-universal property of
$\Bmin{\Lambda}$ implies that there is a homomorphism $\phi_t :
C^*(t)^\alpha \to \Csmin{\Lambda}^\beta$ which takes $t_\mu
t^*_\nu$ to $\Smin_\mu \Smin^*_\nu$ for all $(\mu,\nu) \in
\bal{\Lambda}$. In particular, $\phi_t \circ \pi_t|_{\Tt
C^*(\Lambda)^\delta} = \pi_{\Smin}|_{\Tt C^*(\Lambda)^\delta}$.
Hence~\eqref{eq:composite image zero} implies that
$\pi_\Smin(\Phi^{\delta}(a^*a)) = 0$.

Since $\pi_{\Smin}$ intertwines $\delta$ and $\beta$, it
follows that $\Phi^\beta(\pi_{\Smin}(a^*a)) = 0$. As $\beta$ is
a normal coaction, we deduce that $\pi_{\Smin}(a^*a) = 0$, and
hence that $\pi_{\Smin}(a) = 0$ as required.

The uniqueness assertion follows from an argument identical to
the one establishing uniqueness of $\Bmin{\Lambda}$. It remains
to show that $\psi_t$  is injective if and only if $\alpha$ is
normal and
\[\textstyle
\prod_{\alpha \in E} (t_\mu t^*_\mu - t_{\mu\alpha} t^*_{\mu\alpha}) = 0
    \quad\text{ for every $\mu \in \Lambda$ and every finite exhaustive $E \subset s(\mu)\Lambda$}.
\]
The ``only if'' direction is clear because $\beta$ is normal
and each $\prod_{\alpha \in E} (\Smin_\mu\Smin^*_\mu -
\Smin_{\mu\alpha} \Smin^*_{\mu\alpha}) = 0$. For the ``if''
direction, suppose that the above two conditions are satisfied.
Then the final assertion of Theorem~\ref{thm:Bmin couniversal}
and the claim above imply that $\psi_t$ restricts to an
isomorphism of $C^*(t)^\alpha$. That $\alpha$ is normal implies
that $\Phi^\alpha : C^*(t) \to C^*(t)^\alpha$ is faithful on
positive elements; hence
\[
\psi_t(a) = 0
    \implies \Phi^\beta(\psi_t(a^*a)) = 0
    \implies \psi_t(\Phi^\alpha(a^*a)) = 0
    \implies \Phi^\alpha(a^*a) = 0
    \implies a = 0.
\]
Thus $\psi_t$ is injective.
\end{proof}

The following corollary shows that if $P = \NN^k$, then our
$P$-graph $C^*$-algebra coincides with the $k$-graph
$C^*$-algebra of \cite{RSY2004}.

\begin{cor}\label{lem:k-graphs}
Let $\Lambda$ be a finitely aligned $k$-graph. Then the
co-universal $C^*$-algebra $\Csmin{\Lambda}$ obtained from
Theorem~\textup{\ref{thm:couniversal alg of Lambda}} by regarding
$\Lambda$ as an $\NN^k$-graph is canonically isomorphic to the
$k$-graph $C^*$-algebra of \cite{RSY2004}.
\end{cor}
\begin{proof}
By definition, $C^*(\Lambda)$ is generated by a representation
$t$ of $\Lambda$. Corollary~4.3 of \cite{RSY2004} shows that
$t_v \not= 0$ for all $v \in \Lambda^0$, and that there is an
action $\gamma$ of $\TT^k$ on $C^*(\Lambda)$ satisfying
$\gamma_z(t_\lambda) = z^{d(\lambda)}t_\lambda$ for all
$\lambda$. Under action-coaction duality, $\gamma$ determines a
coaction $\alpha$ of $\ZZ^k$ satisfying $\alpha(t_\lambda) =
t_\lambda \otimes U_{d(\lambda)}$ for all $\lambda \in
\Lambda$. The coaction $\alpha$ is normal because $\ZZ^k$ is
abelian and hence amenable. Moreover, the $t_\lambda$ satisfy
\[
\prod_{\lambda \in E} (t_v - t_\lambda t^*_\lambda) = 0\quad\text{ for all $v \in \Lambda^0$ and finite exhaustive $E \subset v\Lambda$}
\]
by definition of $C^*(\Lambda)$ (see
\cite[Definition~2.5]{RSY2004}). The result therefore follows
from Theorem~\ref{thm:couniversal alg of Lambda}.
\end{proof}

The following corollary shows that our construction is
compatible with inclusions of quasi-lattice ordered groups. We use it in Example~\ref{eg:CLSV} below.

\begin{cor}\label{cor:subsemigroup}
Let $(G,P)$ be a quasi-lattice ordered group, and let $(H,Q)$
be a subgroup; that is, $H \le G$, $Q \le P$, the order on $H$
agrees with that on $G$, and $Q$ is closed under taking least
upper bounds in $P$. Suppose that $Q$ is hereditary in the sense that if $p,q \in P$ with $pq \in
Q$, then $p,q \in Q$. Let $\iota : H \to G$ be the inclusion
map. Let $\Lambda$ be a finitely aligned $Q$-graph. Define
$\Lambda_P$ to be a copy of $\Lambda$ endowed with the degree
map $d_P : \Lambda^P \to P$ given by $d_P = \iota \circ d$.
Then $\Lambda_P$ is a finitely aligned $P$-graph, and
$\Csmin{\Lambda_P} \cong \Csmin{\Lambda}$.
\end{cor}
\begin{proof}
To see that $\Lambda_P$ is a finitely aligned $P$-graph, the
only difficulty is checking the factorisation property, and
this follows from the assumption that $pq \in Q$ forces $p,q
\in Q$.

It is routine to verify that every representation of
$\Lambda_P$ is a representation of $\Lambda$ and vice versa.
Both $\Csmin{\Lambda}$ and $\Csmin{\Lambda_P}$ carry coactions
of $G$ satisfying $\Smin_\lambda \mapsto \Smin_\lambda \otimes
U_{d(\lambda)}$ for all $\lambda$ (the coaction of $G$ on
$\Csmin{\Lambda}$ is the inflation of the canonical coaction of
$H$). Theorem~\ref{thm:couniversal alg of Lambda} implies that
the co-universal representations of both algebras consist of
nonzero partial isometries, so the co-universal properties of
the two algebras yield mutually inverse homomorphisms between
them.
\end{proof}

\section{Examples}

\subsection{Spielberg's $C^*$-algebras of hybrid graphs}

For the following discussion, we need to recall the
cartesian-product graph introduced by Kumjian and Pask
\cite{KP2000}. Given a $k$-graph $\Lambda$ and an $l$-graph
$\Gamma$, the cartesian-product graph is the $(k+l)$-graph
which is equal to the cartesian product $\Lambda \times
\Gamma$, with pointwise operations and structure maps. By
\cite[Corollary~3.5(iv)]{KP2000}, $C^*(\Lambda \times \Gamma)
\cong C^*(\Lambda) \otimes C^*(\Gamma)$.

We recall the construction of a hybrid graph
\cite[Definition~2.1]{Spielberg2007}. A warning: as frequently
happens when treating constructions involving directed graphs
using ideas based on $k$-graphs, it is easiest to reverse the
directions of the edges from \cite{Spielberg2007}. Let $D$ be
the following directed graph:

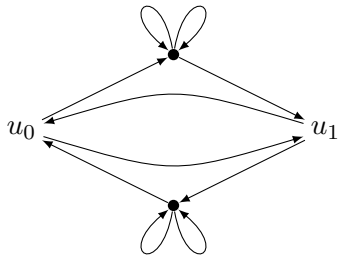
\[
\begin{tikzpicture}[scale=2]
    \node[circle, inner sep=1.5pt] (u0) at (-1,0) {\small$u_0$};
    \node[circle, inner sep=1.5pt] (u1) at (1,0) {\small$u_1$};
    \node[circle, fill=black, inner sep=1.5pt] (a0) at (0,0.5) {};
    \node[circle, fill=black, inner sep=1.5pt] (a1) at (0,-0.5) {};
    \draw[-latex] (u0)--(a0);
    \draw[-latex] (a0)--(u1);
    \draw[-latex] (u1)--(a1);
    \draw[-latex] (a1)--(u0);
    \draw[-latex] (u0) .. controls (0,-0.3) .. (u1);
    \draw[-latex] (u1) .. controls (0,0.3) .. (u0);
    \draw[-latex] (a0) .. controls +(80:0.5cm) and +(40:0.5cm) .. (a0);%
    \draw[-latex] (a0) .. controls +(100:0.5cm) and +(140:0.5cm) .. (a0);%
    \draw[-latex] (a1) .. controls +(280:0.5cm) and +(320:0.5cm) .. (a1);%
    \draw[-latex] (a1) .. controls +(260:0.5cm) and +(220:0.5cm) .. (a1);%
\end{tikzpicture}
\]

Fix, for the rest of the section, irreducible directed graphs
$E_0, E_1, F_0$ and $F_1$ each containing at least one infinite
receiver. We fix infinite receivers $v_i\in E_i^0$ and $w_i\in
F_i^0$, and we attach the $2$-graphs $E_i\times F_i$ to $D$ by
identifying $u_i\in D^0$ with $(v_i,w_i)\in E_i^0\times F_i^0$.
We call the resulting object the \emph{hybrid graph}. The range
and source maps coming from $D$ and each $E_i \times F_i$
extend to range and source maps $r$ and $s$ on the hybrid
graph.

A {\em finite path} in the hybrid graph is a finite string
$\mu_1\dots\mu_k$, where
\begin{enumerate}
\item $\mu_j\in D^*\sqcup\big(\sqcup_{i=0,1}E_i^*\times F_i^*\big)$ for each $1\le j\le k$;
\item $s(\mu_j)=r(\mu_{j+1})$ for each $1\le j< k$; and
\item\label{it:alternating type} $\mu_j\in
    D^*\Longleftrightarrow \mu_{j+1}\in
    \sqcup_{i=0,1}E_i^*\times F_i^*$ for each $1\le j< k$.
\end{enumerate}
We say that paths $\mu,\nu \in
D^*\sqcup\big(\sqcup_{i=0,1}E_i^*\times F_i^*\big)$ are of
\emph{different type} if one of them belongs to $D^*$ and the other to $E_i^*\times F_i^*$.

The range and source maps extend naturally to finite paths:
$r(\mu_1 \dots \mu_k) := r(\mu_1)$ and $s(\mu_1 \dots \mu_k) :=
s(\mu_k)$. Let $l : D^* \to \NN$ be the length function, and
let $l : E^*_i \times F^*_i$ be the standard degree function on
the cartesian-product graph; that is, $l(\alpha,\beta) :=
(l(\alpha), l(\beta))$. Denote by $\Lambda$ the set of all
finite paths in the hybrid graph. Then $\Lambda$ is a category
under concatenation. Define $d : \Lambda \to \NN^2 * \NN$ by
defining $d(\mu_i\dots\mu_k)$ to be the word $l(\mu_1)\dots
l(\mu_k)$.

\begin{lemma}\label{lem:Spielberg is P-graph}
The pair $(\Lambda,d)$ described above is an $(\NN^2 *
\NN)$-graph. Moreover, given finite paths $\mu = \mu_1 \dots
\mu_m$ and $\nu = \nu_1 \dots \nu_n$ with $m \le n$, we have
\begin{equation}\label{eq:MCE for paths in hyb gp}
\MCE(\mu,\nu) =
    \begin{cases}
        \{\nu\} &\text{ if $n > m$, $\mu_i = \nu_i$ for $i < m$,}\\
        &\text{\quad and $\nu_m = \mu_m\nu_m'$}\\
        (\mu_1\dots\mu_{m-1})\MCE(\mu_m,\nu_m) &\text{ if $n = m$ and $\mu_i = \nu_i$ for $i < m$} \\
        \emptyset &\text{ otherwise.}
    \end{cases}
\end{equation}
In particular, $\Lambda$ is finitely aligned.
\end{lemma}
\begin{proof}
We must first show that $(\Lambda,d)$ satisfies the
factorisation property. Indeed, suppose that $d(\lambda) = wx$
where $w,x \in \NN^2 * \NN$. Write $w = w_1 \dots w_m$ and $x =
x_1 \dots x_n$ where each $w_i,x_j \in \NN^2 \sqcup \NN$ and
$w_i \in \NN^2 \implies w_{i+1} \in \NN$ and similarly for $x$.
By definition of $d$, we have $\lambda = \mu_1 \dots \mu_{m-1}
\gamma \nu_2 \dots \nu_n$ where $d(\mu_i) = w_i$ and $d(\nu_i)
= x_i$ for all $i$, and $d(\gamma) = w_mx_1$. If $w_m \in
\NN^2$ and $x_1 \in \NN$ or vice-versa, then the definition of
$d$ again forces $\gamma = \mu_m \nu_1$ where $d(\mu_m) = w_m$
and $d(\nu_1) = x_1$. If, instead, we have both $w_m$ and $x_1$
in $\NN^2$ or both $w_m$ and $x_1$ in $\NN$, then the
factorisation property in either $\bigcup_{i=0,1} E^*_i \times
F^*_i$ or $D^*$ implies that we can factorise $\gamma$ uniquely
as $\gamma = \mu_m \nu_1$ where $d(\mu_m) = w_m$ and $d(\nu_1)
= x_1$. In particular, $\mu := \mu_1 \dots \mu_m$ and $\nu :=
\nu_1 \dots \nu_n$ are the unique paths such that $d(\mu) = w$,
$d(\nu) = x$, and $\lambda = \mu\nu$. Hence $(\Lambda,d)$
satisfies the factorisation property as claimed. In particular,
$(\Lambda,d)$ is a $(\NN^2*\NN)$-graph.

We must establish \eqref{eq:MCE for paths in hyb gp}. To see this, fix $\mu, \nu \in
\Lambda$. First
observe that if $d(\mu) \vee d(\nu) = \infty$, then
$\MCE(\mu,\nu)$ is empty, so we may assume that $d(\mu) \vee
d(\nu) < \infty$. Let $w := d(\mu)$ and $x = d(\nu)$, and write
$w = w_1 \dots w_m$ and $x = x_1 \dots x_n$ so that each $w_i,
x_j \in \NN^2 \sqcup \NN$, and $w_i \in \NN^2 \implies w_{i+1}
\in \NN$ and likewise for $x$. Write $\mu = \mu_1 \dots \mu_m$
and $\nu = \nu_1 \dots \nu_n$ with $d(\mu_i) = w_i$ and
$d(\nu_i) = x_i$. Without loss of generality, assume that $m
\le n$. By definition of the free product, since $w \vee x
\not= \infty$, we must have $w_i = x_i$ for $i < n$. By the
factorisation property, we have
\begin{equation}\label{eq:MCEs in spielberg}
\MCE(\mu,\nu) =
\begin{cases}
(\mu_1 \dots \mu_{m-1})\MCE(\mu_m, \nu_m\dots\nu_n) & \text{if $\mu_i=\nu_i$ for  $i<m$}\\
\emptyset & \text{otherwise}.
\end{cases}
\end{equation}
So suppose $\mu_i=\nu_i$ for $i<m$. We must consider two cases: either $m = n$ or $m < n$. First
suppose that $m = n$. Then~\eqref{eq:MCEs in spielberg} gives
$|\MCE(\mu,\nu)| = |\MCE(\mu_m, \nu_m)| < \infty$ because each
of $D^*$ and $\sqcup_{i=0,1}E_i^*\times F_i^*$ is finitely
aligned. Now suppose that $m < n$. For any $\alpha \in
\Lambda$, the element $d(\nu_m\dots\nu_n\alpha)$ has the form
$x_1 \dots x_n y_1 \dots y_l$ for some $y_i$, and it follows in
particular, that $d(\mu_m\beta) = d(\nu_m\dots\nu_n\alpha)$
forces $w_m = d(\mu_m) \le d(\nu_m) = x_m$. The
factorisation property in either $D^*$ or $\sqcup_{i=0,1}E_i^*\times F_i^*$ then implies that $\nu_m = \mu_m\nu_m'$. Hence $\nu$
is the unique element of $\MCE(\mu,\nu)$. This shows that
$\Lambda$ is finitely aligned as claimed.
\end{proof}

Recall from \cite[Definition~3.3]{Spielberg2007} that the
$C^*$-algebra $\Theta$ associated to the hybrid graph is the
universal $C^*$-algebra generated by mutually-orthogonal
projections $\{S_v : v \in \Lambda^0\}$ and partial isometries
$\big\{S_\nu : \nu \in D^1 \sqcup \big(\bigsqcup_{i=0,1} (E_i^1
\times F_i^0) \sqcup (E_i^0 \times F_i^1)\big)\big\}$
satisfying conditions (i)--(v) of
\cite[Definition~3.3]{Spielberg2007}:
\begin{itemize}
\item[(i)] Each $S_v$ is a projection and each $S_\nu$ is a
    partial isometry.
\item[(ii)] For each $v \in E^0_i$ the projections
    $\{S_{(v,w)} : w \in F^0_i\}$ and partial isometries
    $\{S_{(v,f)} : f \in F^1_i\}$ satisfy the Cuntz-Krieger
    relations for the graph $F_i$.
\item[(ii')] For each $w \in F^0_i$ the projections
    $\{S_{(v,w)} : v \in E^0_i\}$ and partial isometries
    $\{S_{(e,w)} : e \in E^1_i\}$ satisfy the Cuntz-Krieger
    relations for the graph $E_i$.
\item[(iii)] The projections $\{S_v : v \in D^0\}$ are
    mutually orthogonal, and the partial isometries $\{S_e
    : e \in D^1\}$ satisfy $S^*_e S_e = S_{s(e)}$ for all
    $e \in D^1$, and
    \[
    \sum_{r(e) = v} S_e S^*_e \le S_v\text{ for all $v \in D^0$, with equality if $v \not\in \{u_0, u_1\}$}.
    \]
\item[(iv)] if $e \in D^1$ and $f \in
    \big(\bigsqcup_{i=0,1} ((E_i^1 \times F_i^0) \sqcup (E_i^0
    \times F_i^1))\big)$, then $S^*_e S_f = 0$.
\item[(v)] For $e \in E^1_i$ and $f \in F^1_i$, we have
    \begin{align*}
    S_{(e,r(f))}S_{(f,s(e))} &= S_{(f,r(e))}S_{(e,s(f))},\text{ and}\\
    S^*_{(e,r(f))} S_{(r(e),f)} &= S_{(s(e),f)} S^*_{(e,s(f))}.
    \end{align*}
\end{itemize}

In what follows, we extend the generating family in Spielberg's
$C^*$-algebra $\Theta$ to a representation $S$ of the
associated $(\NN^2 * \NN)$-graph $\Lambda$. This should not be
confused with the co-universal representation of $\Lambda$ in
$\Csmin(\Lambda)$.

\begin{prop}\label{isom of algebras}
The $C^*$-algebra $\Theta$ defined above is isomorphic to
$\Csmin{\Lambda}$.
\end{prop}
\begin{proof}
A vertex in $\Lambda$ is an element $v\in
\Lambda^{(0,0)}\cup\Lambda^{0}$, and an edge in $\Lambda$ is an
element
$\nu\in\Lambda^{(1,0)}\cup\Lambda^{(0,1)}\cup\Lambda^{1}$.
Given $\mu \in D^*$, say $\mu = d_1d_2 \dots d_n$ with each
$d_i \in D^1$, we define $S_\mu = S_{d_1} \cdots S_{d_n}$. For
$(\mu,\nu) \in E_i^* \times F_i^*$, with $\mu = e_1 \dots e_m$
and $\nu = f_1 \dots f_n$ with each $e_j \in E^1_i$ and each
$f_j \in F^1_i$, define
\[
S_{(\mu,\nu)} := S_{(e_1, r(\nu))} \dots S_{(e_m,r(\nu))} S_{(s(\mu),f_1)} \dots S_{(s(\mu), f_n)}.
\]
Now for a finite path $\mu = \mu_1 \dots \mu_m$ in $\Lambda$,
define $S_\mu := S_{\mu_1} \dots S_{\mu_m}$.

We claim that $\mu \mapsto S_\mu$ determines a representation
of $\Lambda$ in $\Theta$. Condition~(T1) of
Definition~\ref{repn of Lambda} follows immediately from
(i)--(iii) of \cite[Definition~3.3]{Spielberg2007}. For the
other conditions, let $\lambda,\mu\in\Lambda$ with
$\mu=\mu_1\dots\mu_m$ and $\nu=\nu_1\dots\nu_n$. Suppose
$s(\mu)=r(\nu)$. So $s(\mu_m)=r(\nu_1)$. If $\mu_m$ and $\nu_1$
are of different type, then $S_{\nu}S_{\mu}=S_{\nu\mu}$ follows
immediately. If $\mu_m$ and $\nu_1$ are the same type, then the Cuntz-Krieger relations of $E_i\times F_i$ and
Toeplitz-Cuntz-Krieger relations of $D$ imply that
$S_{\lambda_j}S_{\mu_1}=S_{\nu_j\mu_1}$. We then have
\[
S_{\nu\mu}
    = S_{\nu_1}\dots S_{\nu_j}S_{\mu_1}\dots S_{\mu_k}
    = S_{\nu_1}\dots S_{\nu_j\mu_1}\dots S_{\mu_k}
    = S_{\nu_1\dots(\nu_j\mu_1)\dots\mu_k}
    = S_{\nu\mu},
\]
and so condition~(T2) is satisfied.

Fix a finite path $\mu = \mu_1 \dots \mu_m$ in $\Lambda$. The
Cuntz-Krieger relations of $E_i\times F_i$, and
Toeplitz-Cuntz-Krieger relations of $D$ imply that
$S_{\mu_j}^*S_{\mu_j}=S_{s(\mu_j)}$ for all $1\le j\le m$. An
inductive argument then gives $S_{\mu}^*S_{\mu}=S_{s(\mu)}$,
so~(T3) is satisfied.

The proof of~(T4) is tedious, so we set it aside as a
claim.

\smallskip\noindent\textbf{Claim.} The map $\mu \mapsto S_\mu$ satisfies~(T4).

To prove this claim, fix $\mu,\nu \in \Lambda$. We must show
that
\begin{equation}\label{eq:Speilberg T4}
S_\mu S^*_\mu S_\nu S^*_\nu = \sum_{\lambda \in \MCE(\mu,\nu)} S_\lambda S^*_\lambda.
\end{equation}
Write $\mu = \mu_1 \dots \mu_m$ and $\nu = \nu_1 \dots \nu_n$.
We may assume without loss of generality that $m \le n$. First
suppose that $\mu_1$ and $\nu_1$ are of different type. Then~(iv) ensures that
\[
    S_\mu S^*_\mu S_\nu S^*_\nu = S_\mu S^*_{\mu_2\dots\mu_m} S^*_{\mu_1} S_{\nu_1} S_{\nu_2 \dots \nu_n} S^*_\nu = 0.
\]
Lemma~\ref{lem:Spielberg is P-graph} implies that
$\MCE(\mu,\nu) = \emptyset$ also, so~(T4) is satisfied.

Now suppose that $\mu_1$ and $\nu_1$ have the same type.
Condition~(\ref{it:alternating type}) then implies that $\mu_i$
and $\nu_i$ have the same type for $i \le m$. If $l \le m$
satisfies $\mu_i = \nu_i$ for all $i < l$, then repeated
applications of~(T3) show that
\[
    S_\mu S^*_\mu S_\nu S^*_\nu = S_\mu S^*_{\mu_l\dots\mu_m} S_{\nu_l \dots \nu_n} S^*_\nu,
\]
and then conditions~(ii), (ii')~and~(iii) imply that
\begin{equation}\label{eq:reduced T4}
    S_\mu S^*_\mu S_\nu S^*_\nu
        = \sum_{\mu_l\alpha = \nu_l\beta \in \MCE(\mu_l,\nu_l)} S_\mu S^*_{\mu_m} \dots S^*_{\mu_{l+1}}
            S_\alpha S^*_\beta  S_{\nu_l}\dots S_{\nu_n} S^*_\nu.
\end{equation}
If $\mu_l \not= \nu_l$ for some $l < m$, then $\mu_l\alpha
= \nu_l\beta \in \MCE(\mu,\nu)$ forces at least one of
$d(\alpha),d(\beta) > 0$ and then~(iv) forces one of
$S^*_{\mu_{l+1}} S_\alpha$ and $S^*_\beta S_{\nu_{l+1}}$ to be
equal to zero. Hence $\mu_l \not= \nu_l$ for some $l < m$
forces $S_\mu S^*_\mu S_\nu S^*_\nu = 0$. Since
Lemma~\ref{lem:Spielberg is P-graph} implies that
$\MCE(\mu,\nu) = \emptyset$ unless $\mu_i = \nu_i$ for all $i <
m$ we have now established~\eqref{eq:Speilberg T4} whenever
$\mu_l \not= \nu_l$ for some $l < m$.

So we suppose that $\mu_l = \nu_l$ for all $l < m$, and
consider two cases: $m = n$ or $m < n$. Suppose first that $m =
n$. Then~\eqref{eq:reduced T4} reduces to
\begin{align*}
    S_\mu S^*_\mu S_\nu S^*_\nu
        &= \sum_{\mu_m\alpha = \nu_m\beta \in \MCE(\mu_m,\nu_m)} S_\mu S_\alpha S^*_\beta S^*_\nu \\
        &= \sum_{\mu_m\alpha = \nu_m\beta \in \MCE(\mu_m,\nu_m)} S_{\mu_1\dots\mu_{m-1}} S_{\mu_m\alpha} S^*_{\nu_m\beta} S^*_{\nu_1 \dots \nu_{m-1}}
\end{align*}
and since Lemma~\ref{lem:Spielberg is P-graph} gives
$\MCE(\mu,\nu) = \mu_1\dots\mu_{m-1}\MCE(\mu_m,\nu_m)$, this
establishes~\eqref{eq:Speilberg T4} in the case $m = n$. Now
suppose that $m < n$. Suppose that $\nu_m \not= \mu_m\nu'_m$.
Then $\MCE(\mu,\nu) = \emptyset$ by Lemma~\ref{lem:Spielberg is
P-graph}. Also, $d(\beta) > 0$, and since $\beta$ and
$\nu_{m+1}$ are of different type, condition~(iv) again gives
$S^*_\beta S_{\nu_{m+1}} = 0$. Hence $S_\mu S^*_\mu S_\nu
S^*_\nu = 0 = \sum_{\lambda \in \MCE(\mu,\nu)} S_\lambda
S^*_\lambda$, establishing~\eqref{eq:Speilberg T4} in the case
$m < n$ and $\nu_m \not= \mu_m\nu'_m$. Finally, suppose that
$\nu_m = \mu_m\nu'_m$. Then $\MCE(\mu,\nu) = \{\nu\}$ by
Lemma~\ref{lem:Spielberg is P-graph}, and $\MCE(\mu_m,\nu_m) =
\{\nu_m\}$. Hence~\eqref{eq:reduced T4} reduces to
\[
S_\mu S^*_\mu S_\nu S^*_\nu = S_\mu S_{\nu_m'} S^*_{s(\nu_m)} S_{\nu_{m+1} \dots \nu_n} S^*_\nu = S_\nu S^*_\nu.
\]
We have now established~\eqref{eq:Speilberg T4} in all possible
cases. This proves the claim.%

The claim completes the proof that $\mu \mapsto S_\mu$ is a
representation of $\Lambda$ in $\Theta$.

We denote by $e$ the identity of $\ZZ^2*\ZZ$. Straightforward
calculations show that $\{S_v\otimes 1_{C^*(G)}:
v\in\Lambda^{(0,0)}\cup\Lambda^0\}\cup\{S_\nu\otimes
U_{d(\nu)}:
\nu\in\Lambda^{(0,1)}\cup\Lambda^{(1,0)}\cup\Lambda^{1}\}$ is a
set of projections and partial isometries in $\Theta\otimes
C^*(\ZZ^2*\ZZ)$ satisfying conditions (i)--(v) of
\cite[Definition~3.3]{Spielberg2007}. The universal property of
$\Theta$ then gives a $*$-homomorphism $\alpha:\Theta\to
\Theta\otimes C^*(\ZZ^2*\ZZ)$ such that
$\alpha(S_{\mu})=S_{\mu}\otimes U_{d(\mu)}$.

We show that $\alpha$ is a coaction. It is straightforward to
check that $\alpha$ satisfies the coaction identity on
generators of $\Theta$. Since increasing finite sums $P_F :=
\sum_{v \in F} S_v$ where $F \subset \Lambda^{(0,0)} \cup
\Lambda^0$ form an approximate identity for $\Theta$ such that
$\alpha(P_F) = P_F \otimes 1$ for all $F$, the homomorphism
$\alpha$ is nondegenerate, and it follows (see
\cite[Remark~A.22(3)]{EKQR2006}) that it is a coaction.

Theorem~\ref{thm:couniversal alg of Lambda} now yields a
surjective homomorphism $\psi_S:\Theta\to \Csmin{\Lambda}$.
Corollary~3.19 of \cite{Spielberg2007} implies that $\Theta$ is
simple, and it follows that $\psi_S$ is an isomorphism.
\end{proof}

\begin{cor}
Let $A$ be a simple purely infinite nuclear $C^*$-algebra
belonging to the UCT class. Then there exists a finitely
aligned $(\NN^2 * \NN)$-graph $\Lambda$ such that $A$ is stably
isomorphic to $\Csmin{\Lambda}$.
\end{cor}
\begin{proof}
Fix $K$-groups $K_0$ and $K_1$. The argument of
\cite[Theorem~2.2]{Spielberg2007b} (see also
\cite{Szymanski2002a}) shows that there are graphs $E_i, F_i$,
$i = 0,1$, each with a unique infinite receiver, such that the
$C^*$-algebra $\Theta$ of the associated hybrid graph is a
simple, purely infinite nuclear $C^*$-algebra in the UCT class
with $K_*(\Theta) = (K_0, K_1)$. The result therefore follows
from Lemma~\ref{lem:Spielberg is P-graph} and
Proposition~\ref{isom of algebras}.
\end{proof}

\subsection{Other examples}

We now consider two examples discussed in \cite{pp_SY2007} and
\cite{pp_CLSV2009}.

\begin{example}\label{eg:SY}
Consider the product system of \cite[Example~3.16]{pp_SY2007}.
That is, let $G = (\ZZ^2, +)$, let $S = \{0\} \times \NN
\subset \ZZ^2$, and let
\[
P = S \sqcup ((\NN \setminus \{0\}) \times \ZZ).
\]
Then $(G,P)$ is a quasi-lattice ordered group, and the order on
$G$ induced by $P$ is lexicographic order.

We define a $P$-graph $\Lambda$ as follows. As a set, $\Lambda
= \{f_s, g_s : s \in P\}$, and the degree map is given by
$d(f_s) = d(g_s) = s$. Define range and source maps by $r(f_s)
= s(f_s) = f_0$ for all $s \in P$; $r(g_s) = s(g_s) = g_0$ for
$s \in S$; and $r(g_s) = f_0$, $s(g_s) = g_0$ for $s \in P
\setminus S$. So for $s \in P$, the directed graph $(\Lambda^0,
\Lambda^s, r, s)$ has one of two forms:

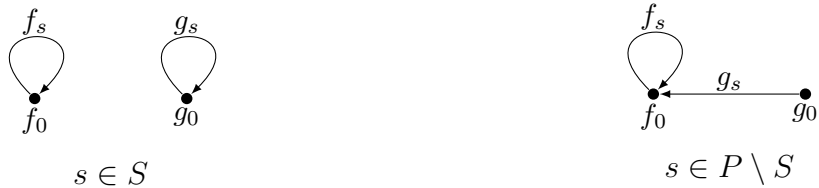
\[
\begin{tikzpicture}
\node[circle,fill=black,inner sep=1.5pt] (f0) at (0,0) {};%
\node[inner sep = 1pt, anchor=north] at (f0.south) {\small$f_0$};%
\node[circle,fill=black,inner sep=1.5pt] (g0) at (2,0) {};%
\node[inner sep = 1pt, anchor=north] at (g0.south) {\small$g_0$};%
\draw[-latex] (f0) .. controls +(135:1.5) and +(45:1.5) .. (f0) node[pos=0.5, anchor=south,inner sep=1pt] {\small$f_s$};%
\draw[-latex] (g0) .. controls +(135:1.5) and +(45:1.5) .. (g0) node[pos=0.5, anchor=south,inner sep=1pt] {\small$g_s$};%
\node at (1,-1) {$s \in S$};
\end{tikzpicture}
\hskip4cm
\begin{tikzpicture}
\node[circle,fill=black,inner sep=1.5pt] (f0) at (0,0) {};%
\node[inner sep = 1pt, anchor=north] at (f0.south) {\small$f_0$};%
\node[circle,fill=black,inner sep=1.5pt] (g0) at (2,0) {};%
\node[inner sep = 1pt, anchor=north] at (g0.south) {\small$g_0$};%
\draw[-latex] (f0) .. controls +(135:1.5) and +(45:1.5) .. (f0) node[pos=0.5, anchor=south,inner sep=1pt] {\small$f_s$};%
\draw[-latex] (g0)--(f0) node[pos=0.5, anchor=south,inner sep=1pt] {\small$g_s$};%
\node at (1,-1) {$s \in P\setminus S$};
\end{tikzpicture}
\]

Define composition by
\begin{align*}
f_s f_t &= f_{s+t}, \\
g_s g_t &= g_{s+t} \quad\text{ for $t \in S$, and}\\
f_s g_t &= g_{s+t} \quad\text{ for $t \in P \setminus S$.}
\end{align*}
It is routine to check that this determines a composition map
which is defined on all composable pairs, that composition is
associative, and that it satisfies the factorisation property.
Hence $\Lambda$ is a $P$-graph. Since $v\Lambda^s$ is a
singleton for all $v \in \Lambda^0$ and $s \in P$, $\Lambda$ is
finitely aligned.

If $X$ is the product system over $P$ whose fibre over $s \in
P$ is the usual graph $C^*$-correspondence (see
\cite[Proposition~3.2]{RS2005}) then $X$ is isomorphic to the
product system described in \cite[Example~3.16]{pp_SY2007}.

Identifying representations of the product system $X$ with
representations of $\Lambda$ as in \cite[Theorem~4.2]{RS2005},
the discussion of \cite[Example~3.16]{pp_SY2007} shows that
every representation $t$ of $\Lambda$ corresponding to a
CNP-covariant representation of $X$ satisfies $t_{g_0} = 0$. In
particular, the universal generating representation of
$\Lambda$ in the algebra $\mathcal{NO}_X$ defined in
\cite{pp_SY2007} does not satisfy $t_v \not= 0$ for all $v \in
\Lambda^0$.

However, since the set $\{g_0, g_1, g_2, \dots\}$ is an
ultrafilter of $\Lambda$, the representation $\Sfilters$ of
Definition~\ref{dfn:Sfilters} satisfies $T_{g_0} \not= 0$ and
hence the generating representation of $\Lambda$ in
$\Csmin{\Lambda}$ consists of nonzero partial isometries. That
is, while the algebra $\mathcal{NO}_X$ does not satisfy
Criterion~(A) of \cite[Section~1.2]{pp_SY2007}, our
$\Csmin{\Lambda}$ does.
\end{example}

\begin{rmk}\label{rmk:SY}
A brief explanation is in order here. The bimodules
$\widetilde{X}_s$ of \cite{pp_SY2007} were intended to model
the sets $\Lambda^{\le s}$ of finite paths in a $k$-graph whose degree
is smaller than $s$ but which cannot be extended nontrivially
in direction $s$ (see \cite{RSY2003}). The Cuntz-Pimsner
covariance condition of \cite{pp_SY2007} was then intended to
model the Cuntz-Krieger relation of \cite{RSY2003}.

In a $P$-graph $\Lambda$ (as opposed to a $k$-graph), the
analogue of $\Lambda^{\le s}$ would be
\[
\Lambda^{\le s} = \{\mu \in \Lambda : d(\mu) \le s, d(\mu) < p \le s \implies s(\mu)\Lambda^{d(\mu)^{-1} p} = \emptyset\}.
\]
In the example above, we have $g_0 \Lambda^s = \{g_s\}$ if $s
\in S$, and $g_0 \Lambda^s = \emptyset$ if $s \in P\setminus
S$. Since $s \in S$ and $t \in P\setminus S$ implies $s \le t$,
it follows that $g_0\Lambda^{\le s} = \emptyset$ for $s \in P
\setminus S$. In particular, the Cuntz-Krieger relations of
\cite{RSY2003}, adapted to $P$-graphs, would force $t_{g_0} =
0$.

The point is that since for $t \in P \setminus S$, the set $\{s
\in P : s \le t\}$ is infinite, the above definition of
$\Lambda^{\le s}$ is inappropriate. Instead, Exel's insight,
when applied to this example, is that the set of ultrafilters
whose elements all have degree smaller than $s$ is the
appropriate analogue of $\Lambda^{\le s}$.
\end{rmk}

\begin{example}\label{eg:CLSV}
Let $\FF_2$ be the free group on two generators, $\FF_2 = \langle
a,b\rangle$, and let $\FF_2^+$ be subsemigroup generated by $a$ and
$b$. Let $\Lambda$ be the $\NN$-graph with $\Lambda^n =
\{e_n\}$ for all $n \in \NN$, and let $\Lambda_{\FF^+_2}$ be
the $\FF^+_2$-graph obtained from the first assertion of
Corollary~\ref{cor:subsemigroup} applied to the embedding of
$\NN$ in $\FF^+_2 = \langle a,b \rangle$ given by $1 \mapsto
a$.

The second assertion of Corollary~\ref{cor:subsemigroup} gives
$\Csmin{\Lambda_{\FF^+_2}} \cong \Csmin{\Lambda}$.
Lemma~\ref{lem:k-graphs} implies that $\Csmin{\Lambda} \cong
C^*(\Lambda)$. It is well-known (see, for example,
\cite[Example~2.14]{Raeburn2005}) that $C^*(\Lambda) \cong
C(\TT)$. By contrast, \cite[Example~3.9]{pp_CLSV2009} shows
that if $X$ is the product system over $\FF^+_2$ corresponding
to $\Lambda_{\FF^+_2}$, then $\mathcal{NO}_X \cong \mathcal{T}$
where $\mathcal{T}$ denotes the Toeplitz algebra. In particular, $\mathcal{NO}_X$ is not co-universal for gauge-compatible representations of $X$.
Moreover, passing to the normalisation $\mathcal{NO}_X^r$ as in
\cite[Section~4]{pp_CLSV2009} doesn't help because the coaction
of $\FF^+_2$ on $\mathcal{NO}_X$ is already normal. In particular, our
construction avoids the pathology arising in this example for
product systems (see \cite[Remark~4.2]{pp_CLSV2009}).
\end{example}

\end{document}